\documentclass[a4paper]{amsart}

\usepackage{type1cm}        
\usepackage{multicol}
\usepackage{newtxtext}       
\usepackage{newtxmath}

\usepackage{amsmath, amsfonts}
\usepackage[T1]{fontenc}
\usepackage[utf8]{inputenc}
\usepackage{enumitem}
\usepackage{xurl}
\usepackage[defaultlines=3,all]{nowidow}
\usepackage{mathrsfs}
\usepackage{tikz}

\usepackage{pgf}
\usepackage{xcolor}
\usepackage{nicematrix}

\usetikzlibrary{calc,cd}

\usepackage[pagebackref,bookmarksopen]{hyperref}
\hypersetup{
	colorlinks=true,
	linkcolor=blue,
	citecolor=magenta,
	urlcolor=cyan,
}
\usepackage{bookmark} 

\DeclareMathOperator{\matArtal}{Mat}
\DeclareMathOperator{\singArtal}{Sing}
\DeclareMathOperator{\discArtal}{disc}
\DeclareMathOperator{\galArtal}{Gal}
\DeclareMathOperator{\cmbArtal}{Comb}
\DeclareMathOperator{\grArtal}{gr}

\newcommand{\abs}[1]{\left\lvert #1\right\rvert}

\newtheorem{theorem}{Theorem}[section]
\newtheorem{proposition}[theorem]{Proposition}
\newtheorem{corollary}[theorem]{Corollary}
\newtheorem{problem}[theorem]{Problem}
\newtheorem{lemma}[theorem]{Lemma}
\newtheorem{ppts}[theorem]{Properties}
\newtheorem{paso}[theorem]{Step}
\newtheorem{cjt}[theorem]{Conjecture}

\theoremstyle{definition}
\newtheorem{definition}[theorem]{Definition}
\newtheorem{example}[theorem]{Example}
\newtheorem{ntc}[theorem]{Step}
\newtheorem{cnt}[theorem]{Construction}

\theoremstyle{remark}
\newtheorem{remark}[theorem]{Remark}

\title[Superisolated singularities]{Superisolated singularities and friends}   
\author[E. Artal]{Enrique Artal Bartolo}

\address{Departamento de Matem\'{a}ticas, IUMA,
Universidad de Zaragoza,
C.~Pedro Cerbuna 12, 50009, Zaragoza, Spain
}

\email{\href{mailto:artal@unizar.es}{artal@unizar.es}}
\urladdr{\url{http://riemann.unizar.es/~artal}}

\thanks{Partially supported by MCIN/AEI/10.13039/501100011033 (grant code: PID2020-114750GB-C31)
and by Departamento de Ciencia, Universidad y Sociedad del Conocimiento del Gobierno de Arag{\'o}n
(grant code: E22\_20R: ``{\'A}lgebra y Geometr{\'i}a'').}

\begin{document}

\begin{abstract}
Superisolated surface
singularities in~$(\mathbb{C}^3,0)$ were introduced by I.~Luengo to prove that
the $\mu$-constant stratum may be singular. The main feature of this 
family is that it can bring information from the projective plane curves
(global setting but smaller dimension) into surface singularities. 
They are simple enough to allow to retreive information and complicated enough
to offer a variety of properties. The so-called L{\^e}-Yomdin singularities
are a generalization which offers a wider catalog of examples. We study 
some properties of these singularities, mainly topological and related with 
the monodromy, and we introduce another family which exploits the same properties
but in the quasi-homogeneous setting.
\end{abstract}

\maketitle


\tableofcontents


\section{Introduction}

This work is devoted to the relationship between two subjects: projective plane curves and germs of surface singularities in $\mathbb{C}^3$.
A projective plane curve is defined by a homogeneous polynomial in three variables, and hence it defines 
a germ of (homogeneous) surface singularity in $(\mathbb{C}^3,0)$. This germ will be in general non-isolated (actually it will be isolated
if and only if the curve is smooth). On the other side, a germ of surface singularity in $(\mathbb{C}^3,0)$ has associated a projective
plane curve, its tangent cone.

Probably the first approach to this idea was done by Y.~Yomdin~\cite{Yomdin74} (see also~\cite{Le80}).
Starting from a reduced homogeneous function $F_d(x,y,z)\in\mathbb{C}[x,y,z]$ of degree~$d$, Yomdin proposes an
\emph{asymptotical} approach using \emph{generic} functions $F_{d+k}(x,y,z)$, $k\in\mathbb{N}$, such that 
the zero locus $V$ of $F_d+F_{d+k}$ has an 
isolated singularity at the origin (there is a more concrete condition which will be made explicit in the text).

Some years later, I.~Luengo used similar ideas in~\cite{Luengo87} to study the smoothness of the $\mu$-constant stratum
of surface singularities in $\mathbb{C}^3$. He defined the family of \emph{superisolated singularities} as the singularities
which are (abstractly) resolved with one blow-up. Actually, they coincide with the Yomdin (L{\^e}-Yomdin from now on)
singularities for $k=1$. While being a particular case, they can be defined in a more intrinsic way  and he made a thorough study of their 
properties. 

The main idea of Luengo is to translate properties of the projective plane curve $C_d$ into properties 
of a superisolated singularity whose tangent cone is~$C_d$, allowing to derive properties of singularities (local, dimension~$2$ embedded in dimension~$3$) from properties of plane curves (global, dimension~$1$ embedded in dimension~$2$).

As a consequence, Luengo proved that the $\mu$-constant stratum of a germ of isolated singularities in $\mathbb{C}^3$ may be non smooth; note that the $\mu$-constant stratum of a germ of an isolated singularity in $\mathbb{C}^2$ is smooth (see e.g.~\cite{wahl:76}). Examples of singular strata were found among superisolated singularities, and the key point of the proof came from the properties
of the tangent cones as projective plane curves.

In the eighties, S.S.-T.~Yau~\cite{yau:89} stated the following conjecture. The abstract topology of the link of a germ of isolated singularity in $\mathbb{C}^3$ and the characteristic polynomial of the monodromy determines the embedded topology of the germ, see Conjecture~\ref{cjt:yau} for a precise statement. The conjecture was verified in important infinite families of singularities. In the nineties, an embedded
resolution of superisolated singularity was found in~\cite{ea:mams}. With this tool the link of the singularity~$V$ can be computed; actually it was already done~\cite{Luengo87} and it can be read from~$C_d$. The geometric monodromy of the Milnor fibration 
of $V$ can also be described as some of its properties. Namely the characteristic polynomial of the action of the monodromy
on the cohomology of the Milnor fibration (previously computed by J.~Stenvens~\cite{Stevens:89}) and also its Jordan form.
How the properties of $C_d$ are related to the above properties of $V$
can be used to discard the conjecture.
The purpose of this work 
is to show this parallelism in order to present how they benefit from each other, and also to present
problems in this field which are open since the nineties.

In the nineties, W.D.~Neumann and J.~Wahl studied the classic Casson's conjecture and refined it in the case of links of 
germs of isolated singularities in $\mathbb{C}^3$, supporting the conjectures by large families of examples. These conjectures
were generalized by A.~N{\'e}methi and L.I.~Nicolaescu to conjectures involving finer invariants such as Seiberg-Witten's.
These conjectures are related to oriented closed $3$-manifolds whose link is a $\mathbb{Q}$-homology sphere. 
In~\cite{FLMN:06} characterized the superisolated singularities of surface whose link is a $\mathbb{Q}$-homology sphere in terms of the tangent cone, leading mainly to rational cuspidal curves. As for the previous
problems, the properties of those curves furnished counterexamples to these conjectures and gave key ideas to reformulate problems.

These are three concrete examples of how superisolated singularities have helped to the study of more general singularities of surface. There is 
survey in 2006~\cite{alm:06}, which develop also these ideas, mainly from an algebraic point of view.
In this work we recall some of these properties but we focus specially on topological properties of the singularities
coming from topological properties of the tangent cones, and we extend the results to L{\^e}-Yomdin singularities,
since having a larger family it is likely that more problems can be attacked in the future.

We may study this problem in higher dimensions, but since there are quite interesting problems as it is, we will focus on surface singularities.
In this work we will present some other generalizations. As in the homogeneous case, a \emph{weighted} projective plane curves is defined 
by a quasi-homogeneous function in three variables, and hence it defines also a germ of singularity in $(\mathbb{C}^3,0)$. In most cases, this germ
has non isolated singularities. Adding \emph{generic} terms of weighted degree greater than the original one, the singularity
becomes isolated and under this genericity conditions, its properties depend on the weighted projective curve. These ideas lead 
in~\cite{ABLM-milnor-number} to the definition of the \emph{weighted L{\^e}-Yomdin singularities}.

In~\S\ref{sec:def}, we introduce the main concepts, notations and first properties of superisolated and L{\^e}-Yomdin singularities.
In~\S\ref{sec:zp} is devoted to the other object of interests, the projective plane curves; we define the notions of combinatorics
and Zariski pairs. The combinatorics of the tangent cone of a superisolated singularity is the main ingredient to determine its link; this is recalled in~\S\ref{sec:at}, where the links of superisolated and L{\^e}-Yomdin singularities are described.
In~\S\ref{sec:mon} we introduce the main topological invariant for germs of hypersurface singularities, the Milnor fibration and its monodromy;
we recall the classical theory and introduce the necessary notations. Monodromy have been defined, but in general its computation is a challenge;
embedded resolutions, described in~\S\ref{sec:res}, provide a model for the monodromy. 
Steenbrink used an embedded resolution to endow the cohomology of the Milnor fiber with a Mixed Hodge Structure. One ingredient of this Mixed Hodge
Structure is the weight filtration which is topological nature; \S\ref{sec:res} ends with the description of this weight filtration
and its relationship with the monodromy, some algebraic aspects have been pushed to\S\ref{sec:peso}.

There are two main goals in~\S\ref{sec:qres}. The first one is to describe an embedded resolution for superisolated and L{\^e}-Yomdin singularities;
this is very long for superisolated singularities and actually undone for weighted L{\^e}-Yomdin singularities and this is the reason of the second goal, the definition
of $\mathbb{Q}$-resolutions which allow to simplify its presentation for superisolated singularities and to effectively construct one such embedded $\mathbb{Q}$-resolutions for L{\^e}-Yomdin singularities. Moreover, we see that Steenbrink theory is also valid for $\mathbb{Q}$-resolutions. This section needs 
the material in \S\ref{sec:quot} and~\S\ref{sec:wbu}.

The next sections are more devoted to problems than two results. The previous study of the monodromy only holds for fields of characteristic~$0$;
\S\ref{sec:z} deals with a method of M.~Escario to compute the monodromy over~$\mathbb{Z}$ and relates it to a question of N.~A'Campo.
In~\S\ref{sec:pb}, we state some open topological problems for superisolated singularities. We end with~\S\ref{sec:wlys} introducing the weighted
L{\^e}-Yomdin singularities where projective plane curves are replaced by weighted projective plane curves. 
There is a still some work to do about these singularities, so they are only introduced; they have been helpful to solve problems
where superisolated and L{\^e}-Yomdin singularities fail, so they enlarge the options of future work.

Summarizing, all these type of singularities are simple enough to allow a thorough study and complicated enough to enlighten deep properties of singularities.

\section{Definitions and first properties}\label{sec:def}

\begin{definition}[Luengo \cite{Luengo87}]
Let $(V,0)\subset(\mathbb{C}^3,0)$ be a germ of isolated surface singularity\index{singularity}. We say that 
$V$ is a \emph{superisolated surface\index{singularity!superisolated} singularity} (SIS) if it is 
resolved after one blow-up.
\end{definition}

More precisely, let $\sigma:(\hat{\mathbb{C}}^3,E)\to(\mathbb{C}^3,0)$ be the blow-up of the origin, i.e.,
\[
\hat{\mathbb{C}}^3:=\{(p,\ell)\in\mathbb{C}^3\times\mathbb{P}^2\mid p\in\ell\}, \quad E:=\{0\}\times\mathbb{P}^1,
\]
and $\sigma$ is the restriction of the first projection; we identify $\mathbb{P}^2$ with the set of lines through~$0$.
The \emph{strict transform} of $V$ is $\hat{V}:=\overline{\sigma^{-1}(V\setminus\{0\})}$, and the restriction
$\hat\sigma:\hat{V}\to V$ is the blow-up of $V$. Hence $V$ is superisolated if $\hat{V}$ is smooth.
There is a simple characterization of SIS.

\begin{proposition}[\cite{Luengo87}]\label{prop:sis_char}
Let us suppose that $V$ is the zero locus of $F(x,y,z)\in\mathbb{C}\{x,y,z\}$. Let
\begin{equation}\label{eq:desc_hom}
F(x,y,z)= F_d(x,y,z) + \sum_{m>d} F_m(x,y,z),\qquad F_d\neq 0,
\end{equation}
be the decomposition in homogeneous forms (the subscript stands for the degree). Let $C_m:=V_\mathbb{P}(F_m)\subset\mathbb{P}^2$
be the projective zero locus of $F_m$ (a maybe non-reduced curve if $F_m\neq 0$).

Then, $V$ is SIS if and only if $\singArtal(C_d)\cap C_{d+1}=\emptyset$.
\end{proposition}

This proposition relates SIS and L{\^e}-Yomdin singularities.

\begin{definition}\label{dfn:lys}
Let $(V,0)\subset(\mathbb{C}^3,0)$ be a germ of isolated surface singularity. We say that 
$V$ (zero-locus of $F(x,y,z)\in\mathbb{C}\{x,y,z\}$) is a \emph{L{\^e}-Yomdin singularity\index{singularity!L{\^e}-Yomdin}} of order~$k$ ($k-$LYS) if
\[
F(x,y,z)= F_d(x,y,z) + \sum_{m\geq d + k} F_m(x,y,z),\qquad F_d\neq 0,
\]
and $\singArtal(C_d)\cap C_{d+k}=\emptyset$.
\end{definition}

\begin{remark}\label{rem:local_eq}
Hence SIS is equivalent to $1$-LYS.
It is interesting to study the above proposition and definition in local coordinates. Let us assume that 
$P=[0:0:1]$ is a point of the tangent cone\index{tangent cone} $C_d$. In particular $[0:0:1]=((0,0,0),[0:0:1])\in\hat{V}\cap E$
and it is the image of the origin of $\mathbb{C}^3$ by the chart 
\begin{equation}\label{eq:chart-blowup}
\begin{tikzcd}[row sep=0,/tikz/column 1/.append style={anchor=base east},/tikz/column 2/.append style={anchor=base west}]
\mathbb{C}^3\rar["\Phi"]&\hat{\mathbb{C}}^3\\
(x,y,z)\rar[mapsto]&{((x z, y z,z),[x:y:1])}.
\end{tikzcd}
\end{equation}
Under these conditions $E$ has $z=0$ as equation while the total transform $\sigma^*V$ has as equation
\[
z^d(f_d(x,y,1) + z^k(f_{d+k}(x,y,1) +z f_{d+k+1}(x,y,1)+\dots))=0.
\]
If we replace $z$ by the product $z_1$ of $z$ and a suitable unit at the origin, the local equation of $E$ is still $z_1=0$
while the local equation of $\hat{V}$ is $f_{d}(x,y,1) +z_1^k=0$. For later use, let us denote $f_d(x,y,1)$ as $f_{d,P}$. 
Then, 
\begin{enumerate}[label=(Loc\arabic{enumi})]
\item If $P$ is a smooth point of $C_d$, then $P$ is a smooth point of $\hat{V}$ and $E\pitchfork\hat{V}$.
\item If $P$ is a singular point of $C_d$ and $k=1$, then $P$ is a smooth point of $\hat{V}$ and $E\not\pitchfork\hat{V}$.
\item If $P$ is a singular point of $C_d$ and $k>1$, then $P$ is a singular point of $\hat{V}$.
\item\label{loc4} In all these cases, the local topological type at $P$ of $E\cup\hat{V}$ is determined by the topological type of $(C_d, P)$ and $k$.
\end{enumerate}
\end{remark}

\begin{remark}
As a first consequence the tangent cone\index{tangent cone} of a LYS is a reduced curve. The idea behind $k$-LYS is to think about 
them as series with $k\gg 0$, while a SIS can be interpreted as a \emph{generic} isolated singularity
with a fixed tangent cone. In the case of non-reduced curves, A.~Melle~\cite{MH00} defined the $t$-singularities\index{singularity!$t$-singularity}
with the same ideas, both in the definition of \emph{generic} singularities with fixed tangent cone
and the definition of an infinite series of singularities. The main feature is that one can stratify
the tangent cone such that \ref{loc4} is essentially true and the topological type is constant at each stratum.
\end{remark}

Let us come back to the discussion of Remark~\ref{rem:local_eq}. The local singularity of $(\hat{V},P)$ is an interesting type of singularity
which has been thoroughly studied.

\begin{definition}
An isolated germ of hypersurface in $\mathbb{C}^3$ is a \emph{$k$-suspension}\index{suspension}
if it admits an equation of the type $z^k-f(x,y)=0$, where $f$ defines an isolated germ of curve singularity.
\end{definition}

Note that any isolated singularity of multiplicity~$2$ is a $2$-suspension, which have been thoroughly studied~\cite{Laufer:71,Laufer:78}.
It is not the case for higher multiplicities.
 \section{Combinatorics and Zariski pairs}\label{sec:zp}

A key component of SIS and $k$-LYS is the tangent cone\index{tangent cone} which is a reduced projective plane curve. Let us introduce some concepts
which will be useful later.

Let $C=C^1\cup\dots\cup C^r$ be the irreducible decomposition of a reduced projective plane curve; we will denote $d$ its degree, and $d_i$
the degree of~$C^i$, i.e., $d=d_1+\dots+d_r$.
A consequence of Zariski-van Kampen theory is that 
\[
H_1(\mathbb{P}^2\setminus C;\mathbb{Z})\cong\mathbb{Z}^{r-1}\oplus\left(\mathbb{Z}/\gcd(d_1,\dots,d_r)\right).
\]
Consider the minimal embedded resolution\index{resolution}
$\rho:Y\to\mathbb{P}^2$ of the points in $\singArtal'(C)$,
the singular points
which are not nodes belonging to distinct irreducible components of $C$.

\begin{definition}
The \emph{combinatorics}\index{combinatorics} of $C$ is the dual graph $\Gamma_C$ of $\rho^{-1}(C)$ decorated by the self-intersections
and where the vertices corresponding to the strict transforms are marked.
\end{definition}

For an irreducible curve the combinatorics determines and is determined by the degree and the topological types of the singularities.
A more sophisticated example is in Figure~\ref{fig:cremona}.

\begin{lemma}
The genera of the irreducible components corresponding to the vertices can be recovered from the combinatorics.
\end{lemma}

\begin{proof}
The non-marked vertices correspond to rational curves. Let $v$ be a marked vertex corresponding to
the strict transform $\hat{C}_i$ of $C_i$.
Let us consider the connected components of non-marked vertices. They correspond to the
points in $\singArtal'(C)$. Knowing $\hat{C}_i^2$ and the contractions we can recover $C_i^2$ which is $d_i^2$.
Note that $\singArtal'(C)\cap C_i\supset\singArtal(C_i)$; we can compute for $p\in\singArtal(C_i)$
\[
\delta_p(C_i)=\frac{\mu_p(C_i)-r_p(C_i)+1}{2}
\]
where $\mu_p(C_i)$ is the Milnor number and $r_p(C_i)$ is the number of local branches. Then,
the genus of $C_i$ equals
\[
\frac{(d_i-1)(d_i-2)}{2}-\sum_{p\in\singArtal(C_i)} \delta_p(C_i)
\]
and the result follows.
\end{proof}

\begin{proposition}\label{prop:local_topology}
Let us $C', C''$ be two plane curves. They have the same combinatorics if and only if there are regular neighbourhoods $N(C'),N(C'')$
and a homeomorphism of pairs $(N(C'),C')\to(N(C''),C'')$.
\end{proposition}

\begin{proof}
Let us consider the graph $\Gamma$ of the combinatorics of a curve $C$. Let us construct a model for
$(N(C),C)$ from $\Gamma$. Consider the minimal resolution $\Gamma$. We can pick up a tubular neighbourhood
$T(C_v)$ for each irreducible component of $\rho^{-1}(C)$ where $v$ is the corresponding vertex. This is
an oriented disk-bundle onto an oriented closed Riemann surface of genus $g_v$ and Euler number $e_v$, i.e.,
determined by $\Gamma$. We can choose these tubular neighbourhoods such that for each double point
of $\rho^{-1}(C)$ the intersection of the two tubular neighbourhoods is a polydisk in $\mathbb{C}^2$ and the curves
are the intersection with the axes. With this construction, the boundary of the
the union of the tubular neighbourhoods is a plumbing manifold~\cite{neu:81} with plumbing graph $\Gamma$.
The contraction of the singular points provide $N(C)$.

If two curves $C',C''$ have the same graph  the above construction provides a homeomorphism $(N(C'),C')\to(N(C''),C'')$. If such a homeomorphism exist, Waldhausen's theory of graph manifolds~\cite{wald:67} combined with
Neumann's plumbing calculus show that the marked graphs must coincide.
\end{proof}

\begin{example}
The combinatorics of an irreducible curve is equivalent to the degree and the topological type of the curves. In general, the marking is not necessary but we have to take it into account because of Cremona transformations.
Consider the curves
\[
C^1: x y z (xy + yz + xz)(x + 2 y) (y + 2 z) (z + 2 x),\
C^2: x y z (x + y + z)(2x +  y) (2y +  z) (2z +  x).
\]

\begin{figure}[ht]
\centering
\begin{tikzpicture}
\begin{scope}
\draw (0,0) circle [radius=1cm];
\coordinate (P) at (90:1);
\coordinate (Q) at (-30:1);
\coordinate (R) at (210:1);
\coordinate (P1) at ($.6*(Q) + .4*(R)$);
\coordinate (Q1) at ($.6*(R) + .4*(P)$);
\coordinate (R1) at ($.6*(P) + .4*(Q)$);

\draw ($1.3*(P) -.3*(Q)$) -- ($1.3*(Q) -.3*(P)$);
\draw ($1.3*(P) -.3*(R)$) -- ($1.3*(R) -.3*(P)$);
\draw ($1.3*(R) -.3*(Q)$) -- ($1.3*(Q) -.3*(R)$);
\draw ($1.4*(P) -.4*(P1)$) -- ($1.4*(P1) -.4*(P)$);
\draw ($1.4*(Q) -.4*(Q1)$) -- ($1.4*(Q1) -.4*(Q)$);
\draw ($1.4*(R) -.4*(R1)$) -- ($1.4*(R1) -.4*(R)$);
\end{scope}

\begin{scope}[xshift=3cm, yscale=-1, yshift=-.25cm]

\coordinate (P) at (90:1);
\coordinate (Q) at (-30:1);
\coordinate (R) at (210:1);
\coordinate (P1) at ($.6*(Q) + .4*(R)$);
\coordinate (Q1) at ($.6*(R) + .4*(P)$);
\coordinate (R1) at ($.6*(P) + .4*(Q)$);

\draw ($1.3*(P) -.3*(Q)$) -- ($1.3*(Q) -.3*(P)$);
\draw ($1.3*(P) -.3*(R)$) -- ($1.3*(R) -.3*(P)$);
\draw ($1.3*(R) -.3*(Q)$) -- ($1.3*(Q) -.3*(R)$);
\draw ($1.4*(P) -.4*(P1)$) -- ($1.4*(P1) -.4*(P)$);
\draw ($1.4*(Q) -.4*(Q1)$) -- ($1.4*(Q1) -.4*(Q)$);
\draw ($1.4*(R) -.4*(R1)$) -- ($1.4*(R1) -.4*(R)$);
\draw (-1.5,-.8) -- (1.5,0);
\end{scope}
\begin{scope}[xshift=6.5cm]
\coordinate (P1) at (90:1);
\coordinate (P2) at (210:1);
\coordinate (P3) at (-30:1);
\coordinate (Q1) at (150:1);
\coordinate (Q2) at (30:1);
\coordinate (Q3) at (-90:1);

\coordinate (R1) at (90:1/2);
\coordinate (R2) at (210:1/2);
\coordinate (R3) at (-30:1/2);

\coordinate (S) at (-1.2,.1);

\draw (P1) -- (Q2) -- (P3) -- (Q3) -- (P2) -- (Q1) -- cycle;
\draw (P1) -- (Q3);
\draw (P2) -- (Q2);
\draw (P3) -- (Q1);
\draw (R1) -- (R2) -- (R3) -- cycle;
\foreach \a in {1,2,3}
{
\draw (R\a) -- (S) -- (P\a);
\fill[gray!50!white] (P\a) circle [radius=.1cm];
\fill[gray] (Q\a) circle [radius=.1cm];
\fill (R\a) circle [radius=.1cm];
}
\fill (S) node[left] {$1$} circle [radius=.1cm];

\end{scope}
\end{tikzpicture}
\caption{Curves $C^1, C^2$ with common (unmarked) graph. The self-intersection of the undecorated black vertices is $0$;
the self-intersection of the gray vertices is $-1$.}
\label{fig:cremona}
\end{figure}
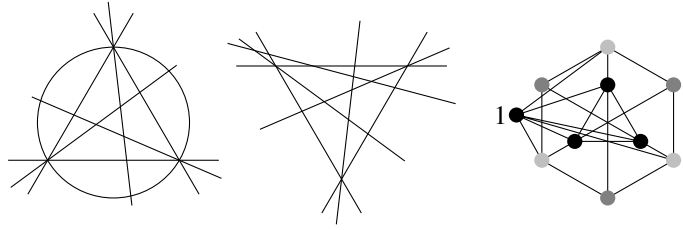
We pass from one to the other via the Cremona transformation $[x:y:z]\mapsto[yz:x z:xy]$.
The black vertices are always marked;
for $C^1$ (resp. $C^2$) the other marked vertices are the darker (resp. lighter) gray ones.
\end{example}

We can interpret Proposition~\ref{prop:local_topology} as follows: two curves with the same combinatorics have the same \emph{local topology}.

\begin{example}
In \cite{zr:29} O.~Zariski studied sextic curves with six ordinary cusps. These are irreducible curves.
He studied a subfamily with a special property, the six cusps were on a conic and showed that
for such a curve $C'$ one has that $\pi_1(\mathbb{P}^2\setminus C')\cong\mathbb{Z}/2*\mathbb{Z}/3$. In this paper
he showed also that if for such a sextic $C''$ the cusps were not in a conic,
then there was no epimorphism $\pi_1(\mathbb{P}^2\setminus C'')\to\mathfrak{S}_3$, in particular,
there is no homeomorphism $(\mathbb{P}^2,C')\to(\mathbb{P}^2,C'')$.
\end{example}

\begin{definition}[\cite{ea:jag}]
Two  reduced projective plane curves form a \emph{Zariski pair}\index{Zariski pair} if they have the same combinatorics but there is no homeomorphism
of $\mathbb{P}^2\to\mathbb{P}^2$ sending one to the another.
\end{definition}

\begin{remark}
Actually, in \cite{zr:29} Zariski did not show the existence of a Zariski pair, since he did not prove the existence of a curve $C''$. He found such sextics in \cite{zr:37} using deformation arguments (no explicit example) and he proved
that his examples have abelian fundamental group. Several decades later, explicit examples appeared in~\cite{Oka:92, ea:jag}.
\end{remark}

\begin{cnt}\label{cnt:cyclic_cover}
Given a curve $C=\{f(x,y,z)=0\}$ of degree~$d$, let 
\[
X_d:=\{[x:y:z:t]\in\mathbb{P}^3\mid t^d-f(x,y,z)=0\},
\begin{tikzcd}[row sep=0,/tikz/column 1/.append style={anchor=base east},/tikz/column 2/.append style={anchor=base west}]
X_d\rar["\tau"]&\mathbb{P}^2\\
{[x:y:z:t]}\rar[mapsto]&{[x:y:z]}
\end{tikzcd}
\]
be the $d$-cyclic cover of $\mathbb{P}^2$ ramified along~$C$. The monodromy of the cover\index{monodromy!of the cover} is generated by
\[
\begin{tikzcd}[row sep=0,/tikz/column 1/.append style={anchor=base east},/tikz/column 2/.append style={anchor=base west}]
X_d\rar["\eta"]&X_d\\
{[x:y:z:t]}\rar[mapsto]&{[x:y:z:\exp\frac{2i\pi}{d} t]}
\end{tikzcd}
\]
Let us recall that if $X'_d$ is any birational smooth model
of $X_d$, then $H^1(X_d;\mathbb{C})$ and $H^1(X'_d;\mathbb{C})$ are isomorphic; moreover,
this isomorphism commutes with the morphisms induced by $\eta$.
This construction is closed to the one of Alexander polynomial in knot theory.
\end{cnt}

\begin{definition}[Libgober \cite{lib:82}]\label{dfn:alexander}
The \emph{Alexander polynomial}\index{Alexander polynomial} of $C$ is the characteristic polynomial of the map induced by 
$\eta$ on $H^1$.
\end{definition}

In modern terms, Zariski proved in~\cite{zr:31} that the Alexander polynomials of $C',C''$ were different, $t^2-t+1$
and $1$, respectively. The Alexander polynomial can be obtained from the fundamental groups
of their complements (which are not isomorphic). The interesting point is that Zariski
computes the Alexander polynomial from some geometric properties of the curves and no fundamental group computation is needed.
The difference in the Alexander polynomials comes from the geometric property of the cusps being or not in a conic.

\begin{definition}
An \emph{Alexander-Zariski} pair\index{Zariski pair!Alexander-Zariski pair} is a Zariski pair where the Alexander polynomials of the curves are distinct.
\end{definition}

Zariski's example is an Alexander-Zariski pair. We know now a lot of Zariski pairs, not all of them are Alexander-Zariski,
sometimes there are geometric properties which characterize them but do not have an effect on the Alexander polynomial.
There are also the so-called \emph{arithmetic Zariski pairs}\index{Zariski pair!arithmetic} where the equations of the members of the pairs live in a number field and they are related by a Galois automorphism of the field, compare with the discussion in~\S\ref{sec:pb}. A more detailed description can be found in~\cite{act:08}, though many new features have appeared since then.

Let us summarize some of the dependence between combinatorics and topological properties.

\begin{ppts}
Let $C$ be a reduced projective planes curve.
\begin{enumerate}[label=\rm(\alph{enumi})]
\item $H_*(\mathbb{P}^2\setminus C;\mathbb{Z})$ depends only on the combinatorics.
\item $\pi_1(\mathbb{P}^2\setminus C)$ does not depend only on the combinatorics.
\item The Alexander polynomial is not a combinatorial invariant.
\end{enumerate}
\end{ppts}
 \section{Abstract topology\index{topology!abstract}}\label{sec:at}

Let us come back to surface singularities.
The main difference between SIS and $k$-LYS for $k>1$ is that the first blowing up does not solve
the singularities of the surface. Nevertheless, since the tangent cone is reduced, whenever it is smooth
the strict transform and the exceptional component are transversal. Let us elaborate these ideas.

Let $V$ as in Definition~\ref{dfn:lys}. The total transform $\sigma^* V$ of $V$ by $\sigma$ is a divisor
which decomposes as $\hat{V}+d E$. Recall that $E=\{0\}\times\mathbb{P}^2\equiv\mathbb{P}^2$ is a projective plane and the tangent cone $C_d$
is $E\cap\hat{V}$. Any irreducible compact divisor $D$ of a variety~$W$ has an \emph{Euler class}\index{Euler class} (or \emph{self-intersection}\index{self-intersection}) $e_W(D)$; this is true also when $W$ is not singular, and in this case $e_W(D)$ can be a rational number. It is a 
linear class of
divisors in $D$ corresponding to the line bundle $\mathscr{O}(D)$ restricted to $D$. 
Sometimes it is also denoted as $E\cdot E
=E^2=(E^2)_{W}$
since it is actually an intersection (the notation for the ambient variety will drop if no ambiguity is expected).

For the exceptional divisor
$E\subset\hat{\mathbb{C}}^3$ of $\sigma$ the Euler class can be detected by the following property. Let $H$ be any divisor in $\mathbb{C}^3$; then $E\cdot\sigma^*H=0$. 

If $H$ is a hyperplane through the origin we have that  $\sigma^* H=E+\hat{H}$, where $\hat{H}$ is the strict transform. It is clear
that 
$\hat{H}\cap E=L$ which is a line (namely $\mathbb{P}(H)$). Hence:
\[
0=E\cdot \sigma^* H =E^2 +E\cdot\hat{H}=E^2 + L\Longrightarrow e(E)=E^2=-L.
\]
The intersection properties of $C_d=\hat{V}\cap E$ in $E\equiv\mathbb{P}^2$ come from Bézout Theorem. 
We use the Euler class to compute the intersection properties of $C_d$ in~$V$.

\begin{proposition}\label{prop:abstract_inter}
Let $V$ be a $k$-LYS. Let $C_d:=C^1+\dots+C^r$ the decomposition of $C_d$ in irreducible divisors
of degrees $d_1,\dots,d_r$ with $d=d_1+\dots+d_r$. Then,
\[
(C_i\cdot C_j)_{\hat{V}}=\frac{d_i\cdot d_j}{k}\text{ if }i\neq j,\quad
(C^i\cdot C^i)_{\hat{V}} =-\frac{d_i\cdot(d-d_i+k)}{k}.
\]
\end{proposition}

\begin{proof}[Guidelines]
The ideas for the proof
are inspirated from \cite{Luengo87} (resp.~\cite{ACMNemethi})
for the SIS (resp.~$k$-LYS) case.
Since $\hat{V}$ is generically transversal to $E$, we recall that
then $(C^j\cdot C_d)_{\hat{V}}= (C^j\cdot e(E))_E= (C^j\cdot L)_E=-d_j$.

Let us start with the case $i\neq j$.
In the SIS case, Luengo~\cite{Luengo87} showed that the intersections of $C^i, C^j$, $i\neq j$ are
isomorphic in both $E$ and $\hat{V}$, and then $(C^i\cdot C^j)_{\hat{V}}= d_i\cdot d_j$.
In \cite{ACMNemethi}, using properties of intersections and coverings, we see that in general $(C^i\cdot C^j)_{\hat{V}}= \frac{d_i\cdot d_j}{k}$.
Combining with the first formula, we obtain the statement.
\end{proof}

By the very definition of SIS, $\hat{\sigma}:\hat{V}\to V$ is a resolution\index{resolution}
of $V$. Note that from Proposition~\ref{prop:abstract_inter}, there is no $(-1)$-curve, hence,
$\hat{\sigma}$ is the minimal resolution of $V$. A minimal log resolution $\tilde\sigma:\tilde{V}\to V$
is obtained combining $\hat{\sigma}$ with the minimal embedded resolution of $\singArtal'(C_d)\subset\hat{V}$. The divisor 
$\tilde\sigma^*(0)$ is a simple normal crossing divisor. 
A regular neighbourhood of $\tilde\sigma^*(0)$ in $\tilde{V}$ can be constructed 
as in Proposition~\ref{prop:local_topology} and its neighbourhood is also a graph
manifold associated to the dual graph $\tilde{\Gamma}$; if $\Gamma$ is 
the graph obtained in that proposition, note that they are equal if we forget the decorations.
More precisely, the decorations of the non-marked vertices coincide for both graphs. If $e_i$
(resp.~$\tilde{e}_i$) is the Euler number of the vertex in $\Gamma$ (resp.~$\tilde{\Gamma}$)
corresponding to $C_i$, then
\[
e_i-\tilde{e}_i=d_i^2 - (-d_i(d-d_i+1)) = d_i(d+1),
\]
since we have to perform the same blow-ups in order to get simple normal crossing divisors in both cases.

\begin{proposition}[\cite{ACMNemethi}]
If $V$ is a $k$-LYS ($k>1$), a log resolution $\tilde{\sigma}:\tilde{V}\to V$ is obtained from $\hat{\sigma}$ by performing 
the embedded resolution of the function $f_{d,P}:(\hat{V},P)\to\mathbb{C}$, $\forall P\in\singArtal(C_d)$.

In particular, for $k\geq 1$, the topology of the link\index{link} of $V$ depends only on $k$ and on the combinatorics of $C_d$. 
\end{proposition}

\begin{proof}
The discussion previous to the statement serves for $k=1$. Let us consider now the case $k>1$.
Recall that $\singArtal(C_d)$ and $\singArtal(\hat{V})$ are the same sets. 

Let $P\in\singArtal(C_d)$, and let us assume that we are in the situation of Remark~\ref{rem:local_eq} where the local equation of $\hat{V}$ at $P=(0,0,0)$ is $z_1^k-f_{d,p}(x, y ,1)=0$. Let us consider a (minimal) embedded resolution of the function $f_{d,P}$ in this singularity; its dual graph
is completely determined by $k$ and the topological type of $C_d$ at $P$. 

The union of the exceptional
components for all $P\in\singArtal(C_d)$ with the strict transform of $f$ is a simple normal crossing divisor
and we obtain a log resolution of $\hat{V}$. Moreover, the decorated dual graph is completely determined by the combinatorics of~$C_d$.
\end{proof}

\begin{remark}\label{rem:abs_res_lys}
There is a simple way to produce a log-resolution of the $k$-suspension\index{suspension} of $f(x,y)=0$. 
Note that the link of the singularity is the cyclic cover of $\mathbb{S}^3$ ramified along the link of $f(x,y)=0$.
There is a general procedure by C.~Safont~\cite{Safont:90} to obtain the plumbing graph of a cover of 
a graph manifold, which can be extended to the case where it is ramified along fibers of the blocks.
When $k=2,3$, this procedure can be simplified. When $k>3$, one needs to use the theory of quotient singularities, see~\S\ref{sec:quot}.
The normalization~$Z$ of the pull-back of the suspension and an embedded resolution of $f=0$ is a singular surface. The singular points
are the preimages of the double points of the divisor of the embedded resolution, and they are the normalizations of singularities
with equations $z^{k}-u^a v^b=0$ which are disjoint unions of quotient singularities, see Example~\ref{ejm:hnk}. How to deal with the intersections numbers in
the singular space~$Z$
is done in~\cite{AMO:2014a}, and for its resolution in Proposition~\ref{prop:min_res}.
\end{remark}

\begin{remark}\label{rem:accion}
There is a global version of the previous remark. Since the topology of a $k$-LYS does not depend on the particular $f_m$, $m\leq d+k$,
as far as $f_{d+k}$ satisfies the condition in Definition~\ref{dfn:lys}, we may assume that $F=f_d(x,y,z)+z^{d+k}$ where $f_d(x,y,0)$
does not have multiple factors. Under this hypothesis $(x,y,z)\mapsto\zeta\cdot(x,y,z)$ preserves $V$ if $\zeta^k=1$. 
This action lifts to $\hat{V}$, the quotient $\hat{V}_k$ is smooth, the preimage of $C_d$ is isomorphic to $C_d$ and if we keep the notation
for the irreducible components of $C_d$ we have
\[
(C_i\cdot C_j)_{\hat{V}_k}=d_i\cdot d_j\text{ if }i\neq j,\quad
(C^i\cdot C^i)_{\hat{V}_k} =-d_i\cdot(d-d_i+k).
\]
Since $\hat{V}$ is a $k$-cyclic cover of $\hat{V}_k$ ramified along $C_d$, a log-resolution of $V$ can be effectively obtained
from an embedded resolution of $C_d\subset\hat{V}_k$.
\end{remark}

As we stated in the introduction, singularities  whose link
is a $\mathbb{Q}$-homology sphere is an important family. SIS and $k$-LYS  with this property can be characterized.

\begin{proposition}\label{prop:lys_qhs}
Let $V$ be a $k$-LYS with tangent cone\index{tangent cone} $C_d$. If $k=1$, the link\index{link} of $V$ is
a $\mathbb{Q}$-homology sphere\index{Q homology sphere@$\mathbb{Q}$-homology sphere} if and only if all these irreducible components are
cuspidal rational curves and if $C_d$ is reducible, then all the components
intersect at only one point.

If $k>1$, besides the above conditions, $\forall P\in\singArtal(C_d)$
the link of its $k$-suspension is also a $\mathbb{Q}$-homology sphere.
\end{proposition}

\begin{proof}
Recall that the condition on the link is equivalent to ask that all the irreducible components
are rational and the dual graph is a tree.
\end{proof}

 \section{On the monodromy}\label{sec:mon}

One of the main invariants in singularity theory is the monodromy\index{monodromy}. 
Let $F:(\mathbb{C}^{n+1},0)\to(\mathbb{C},0)$ be the defining equation of a germ
of isolated singularity of a hypersurface.
Following Milnor~\cite{Milnor:68}
we know that there exists $\varepsilon>0$ such that $F$ is defined in
a neighbourhood of $\overline{\mathbb{B}}^{2n+2}_{\varepsilon}$,
the origin is its unique singular point, and
$V:=F^{-1}(0)$ intersects transversally $\mathbb{S}^{2n+1}_{\varepsilon_0}$, $\forall\varepsilon_0\in(0,\varepsilon]$.
By arguments of compactness, there exists $\eta>0$ such that $F^{-1}(t)\pitchfork\mathbb{S}^{2n+1}_{\varepsilon}$
for all $t\in\overline{\mathbb{D}}^{2}_\eta$.

The Milnor fibration\index{Milnor!fibration} has as representative the map $F^*_{\varepsilon,\eta}$
\begin{equation}
\label{eq:fibration}
\begin{tikzcd}
\overline{\mathbb{B}}^*_{\varepsilon,\eta}:=
\overline{\mathbb{B}}^{2n+2}_{\varepsilon}\cap F^{-1}(\overline{\mathbb{D}}^{2}_\eta\setminus\{0\})
\ar[rr,"F^*_{\varepsilon,\eta}:=F_{|}"]\dar[hookrightarrow]&&
\overline{\mathbb{D}}^{2}_\eta\setminus\{0\}\dar[hookrightarrow]\\
\overline{\mathbb{B}}_{\varepsilon,\eta}:=
\overline{\mathbb{B}}^{2n+2}_{\varepsilon}\cap F^{-1}(\overline{\mathbb{D}}^{2}_\eta)
\ar[rr,"F_{\varepsilon,\eta}:=F_{|}"]&&
\overline{\mathbb{D}}^{2}_\eta
\end{tikzcd}
\end{equation}
and its topological type does not depend on the particular choices. Its homotopy type is the same if
we replace $\overline{\mathbb{D}}^{2}_\eta\setminus\{0\}$ by $\mathbb{S}^1_\eta$
and we know that fibrations on the circle are completely determined by the so-called \emph{geometric
monodromy}\index{monodromy!geometric}. If
\[
\Sigma:=F^{-1}(\eta)\cap\overline{\mathbb{B}}^{2n+2}_{\varepsilon}\quad
\text{(the Milnor fiber\index{Milnor!fiber})}
\]
then the monodromy $\Phi:\Sigma\to\Sigma$ is a homeomorphism, only well-defined up to isotopy (actually it can be chosen a diffeomorphism), such that this fibration
\[
\begin{tikzcd}
{[0,1]}\times\Sigma/(1,z)\sim(0,\Phi(z))\rar&{[0,1]}/0\sim 1\equiv\mathbb{S}^1.
\end{tikzcd}
\]
is isomorphic to $F^*_{\varepsilon,\eta}$.

Moreover Milnor showed that it has the homotopy type of a wedge of $\mu$ $n$-spheres, where $\mu$ (the Milnor number\index{Milnor!number})
is actually the codimension of the Jacobian ideal of~$F$.
Then, $H_n(\Sigma;\mathbb{Z})$ is the only non trivial reduced homology group of~$\Sigma$ and
$\Phi_*:H_n(\Sigma;\mathbb{Z})\to H_n(\Sigma;\mathbb{Z})$ is a well-defined automorphism called the \emph{homological monodromy}\index{monodromy!homological}.
In the same way we can define its dual $\Phi^*:H^n(\Sigma;\mathbb{Z})\to H^n(\Sigma;\mathbb{Z})$, the \emph{cohomological monodromy}\index{monodromy!cohomological}.
Most of the time we will write only \emph{monodromy} and the context will determine the one we refer to.

Though $\Phi_*$ loses a lot of information from $\Phi$, it provides more effective invariants. The first one
is its characteristic polynomial $\Delta$. Resolutions, see~\S\ref{sec:res} and~\S\ref{sec:qres}, provide tools
to compute a representative of $\Phi$. In particular, $\Delta$ can be retrieved from the famous A'Campo's formula\index{A'Campo's formula}~\cite{Acampo:75}.
This polynomial can be computed for $k$-LYS.

\begin{proposition}\label{prop:char}
Let $V$ be a $k$-LYS singularity with tangent cone $C_d$. For $k=1$
\begin{equation}\label{eq:char-sis}
\Delta = \frac{(t^d -1)^{d^2 - 3 d + 3 - \mu(C_d)}}{t-1}
\prod_{p\in\singArtal(C_d)}\Delta_p(t^{d+1}),
\end{equation}
where $\Delta_p$ is the characteristic polynomial of the monodromy\index{monodromy!characteristic polynomial} $\Phi_p$ of $(C_d,p)$.

If $k>1$, then
\begin{equation}\label{eq:char-lys}
\Delta = \frac{(t^d -1)^{d^2 - 3 d + 3 - \mu(C_d)}}{t-1}
\prod_{P\in\singArtal(C_d)}\Delta_P^{(k)}(t^{d+k}),
\end{equation}
where $\Delta_P^{(k)}$ is the characteristic polynomial of $\Phi_P^k$.
In particular,
$\mu =(d - 1)^3 + k\mu(C_d)$.
\end{proposition}
Formula~\eqref{eq:char-sis} was stated without proof in~\cite{Stevens:89}; it was restated and proved in~\cite{ea:mams},
using an embedded resolution and A'Campo's formula. Formula~\eqref{eq:char-lys} was proved in~\cite{GLM:97}, using
an A'Campo-like formula for partial resolutions, and reproved in~\cite{jmm:14}.

\begin{corollary}
The characteristic polynomial of the monodromy of a $k$-LYS depends only on $k$ and on the combinatorics
of its tangent cone.
\end{corollary}

This corollary brings to our attention a conjecture of S.S.-T.~Yau cited in the Introduction.

\begin{cjt}[\cite{yau:89}]\label{cjt:yau}
Let $(V_i,0)\subset(\mathbb{C}^3,0)$, $i=1,2$, two germs of isolated singularity of surface. Let $K_1,K_2$ be their links
and let $\Delta_1,\Delta_2$ be the characteristic polynomials of their monodromies. If $K_1$ is homeomorphic to $K_2$
(they have the \emph{same abstract topology\index{topology!abstract}}) and $\Delta_1=\Delta_2$, then they have the \emph{same embedded topology\index{topology!embedded}}, i.e.,
there is a germ $\Psi:(\mathbb{C}^3,0)\to(\mathbb{C}^3,0)$ of homeomorphism such that $\Psi(V_1)=V_2$.
\end{cjt}

The author supported this conjecture since it holds for large families of examples. The above discussion
shows that if we have two $k$-LYS whose tangent cones have the same combinatorics, then they satisfy the hypotheses
of the conjecture. The characteristic polynomial of the monodromy is an invariant of the conjugacy class of the action
of the monodromy on cohomology. In order to study the conjecture, we want study finer invariants of the monodromy,
namely the Jordan form.

Note also that another consequence of Proposition~\ref{prop:char} is that all the eigenvalues are roots of unity.
Actually a cornerstone in singularity theory gives a sharper statement, the Monodromy Theorem. Let us introduce
some notation in order to give a shorter statement. We need some concepts from~\S\ref{sec:peso}, namely
the weight filtrations of nilpotent endomorphisms.

\begin{definition}
Let $V$ a $\mathbb{Q}$-vector space of finite dimension and let $h:V\to V$ be an automorphism
such that its characteristic polynomial is a product of cyclotomic polynomials. Let $m\in\mathbb{Z}_{>0}$
such that the characteristic polynomial of $h^m$ is a power of $t-1$. Let $W$ be the weight filtration
of the nilpotent endomorphism $\mathbf{1}-h^m$, centered at ~$0$. For $k\geq 0$, the \emph{$k$-characteristic polynomial of $h$}
is the characteristic polynomial $\Delta^{[k]}$ of the isomorphism induced by $h$ on $\grArtal^W_k$ (which coincides
with the one on $\grArtal^W_{-k}$).
\end{definition}

The reader may skip in a first approach\S\ref{sec:peso}. What it is relevant, from Proposition~\ref{prop:weight}\ref{prop:weight_eigen}, 
is that the degree of $\Delta^{[k]}$ is the number of Jordan blocks\index{Jordan blocks} of size~$(k+1)$, and the polynomial itself gives
their eigenvalues.

\begin{theorem}[\cite{br:66}, {\cite[pp. 10--13]{GrDel:72}}]\label{thm:mon}
The eigenvalues of $\Phi$ are roots of unity; moreover, $\Delta^{[m]}=1$, if $m>n$, and $1$ is not a root of~$\Delta^{[n]}$, i.e., the maximal size for Jordan blocks is $n+1$
(and $n$ for the eigenvalue~$1$).
\end{theorem}
 \section{Embedded resolutions}\label{sec:res}

In order to prove \eqref{eq:char-sis} an embedded resolution of a SIS is needed, and this is the main result of~\cite{ea:mams}.
Since an embedded resolution provides a model for the geometric monodromy, much more information can be retrieved.
We recall the general definition.

Let $(V,0)\subset(\mathbb{C}^{n+1},0)$ be a germ of isolated hypersurface singularity defined by $F\in\mathbb{C}\{x_0,x_1,\dots,x_n\}$.

\begin{definition}
An \emph{embedded resolution}\index{resolution!embedded} of $V$ is a proper birational map $\pi:(X,E)\to(\mathbb{C}^{n+1},0)$, $X$ smooth, which is an isomorphism outside the origin and
such that $\pi^*V$ is a simple normal crossing divisor.
\end{definition}

Under these conditions, $E=\pi^{-1}(0)$ is also a simple normal crossing divisor with irreducible components $E_1,\dots,E_r$. The strict
transform of $V$ is $\hat{V}:=\overline{\pi^{-1}(V\setminus\{0\})}$ and the total transform is the simple normal crossing divisor
\begin{equation}\label{eq:total}
\pi^*V = \hat{V} + \sum_{j=1}^r N_j E_j.
\end{equation}
Let us recall A'Campo's formula\index{A'Campo's formula}. For the sake of simplicity let us denote $E_0:=\hat{V}$. We set
for $J\subset\{0,1,\dots,r\}$:
\[
E_J=\bigcap_{j\in J} E_j,\qquad
\mathring{E}_J:=E_J\setminus\bigcup_{j\notin J} E_j.
\]
Let $h$ be a homeomorphism of an $n$-dimensional manifold. We define the zeta function of $h$ as
\[
\zeta_h(t):=\prod_{j=0}^n \det\left(\mathbf{1}-t h^{*,H^j(X;\mathbb{C})}\right)^{(-1)^j},
\]
i.e., the alternating product of the characteristic polynomials of $h$ acting on cohomology. When $V$ is an isolated singularity
of hypersurface, $\zeta_V$ denotes the zeta function of the monodromy.
\begin{theorem}[A'Campo \cite{acampo:03}]
Let  $V$ is an isolated singularity
of hypersurface. With the above notation,
\begin{equation}\label{eq:zeta}
\zeta_V(t)=\prod_{j=1}^r (1-t^{N_j})^{\chi(\mathring{E}_j)},
\end{equation}
where $E_j:=E_{\{j\}}$.
\end{theorem}

In this formula, only the \emph{smooth part} of $\pi^{-1}(V)$ matters, i.e. only 
subsets $J$ with one element, so apparently we have introduced useless notation. 
It is not the case, since, starting from an embedded resolution
J.~Steenbrink constructs in~\cite{Steenbrink:76, Steenbrink:77} (using all the above spaces) a spectral sequence\index{spectral sequence}~$\mathcal{E}$
(see Construction~\ref{cnt:ss} and the discussion after Theorem~\ref{thm:steenbrink} for details)
with the following properties:

\begin{enumerate}
\item The level $\mathcal{E}_1$ is composed by the cohomology of some cyclic covers of the spaces $E_J$, $\# J\leq 1$.
\item It converges to the cohomology of the Milnor fiber.
\item It degenerates at the level $\mathcal{E}_2$.
\item It provides a Mixed Hodge Structure to the cohomology of the Milnor fiber.
\end{enumerate}

Mixed Hodge Structures\index{Mixed Hodge structure}
are related to two filtrations, Hodge\index{filtration!Hodge} (analytic) and weight\index{filtration!weight} (topological), and we are mainly interested in this second one, since it 
contains the information about the Jordan form of the cohomology of the fiber.
We are going to explain how to construct the spectral sequence~$\mathcal{E}$, how to obtain the weight filtration and how
to relate it to the monodromy, and in particular to its Jordan form.

In order to construct the spectral sequence we need to start with the semistable normalization,
see~\cite{Steenbrink:77}.
Let $N$ be any common multiple
of all $N_i$'s and let us consider the following diagram
\[
\begin{tikzcd}
Z\rar["\nu"]&Y\rar["g"]\dar&X\dar["F\circ\pi"]\\
&\mathbb{C}\rar["t\mapsto t^N" below]&\mathbb{C}
\end{tikzcd}
\]
where $Y$ is the pull-back and $\nu$ the normalization. While $X$ is smooth, it is not the case for neither $Y$ nor $Z$, but $Z$
is normal and its singularities are of abelian quotient type.
It is not hard to see that $g$ is an isomorphism from $g^{-1}(E)$ to $E$, but~$\nu$
is a (piecewise) cyclic cover $(g\circ\nu)^{-1}(E)\to E$. 

Let $D:=(g\circ\nu)^{-1}(E)$,
and let us denote $D_j:=(g\circ\nu)^{-1}(E_j)$; note that it may be reducible and not connected.
We define $D_J$ and $\mathring{D}_J$ in the same way.

\begin{proposition}
The map $\nu_J:=\nu_{|D_J}:D_J\to E_J$ is a cyclic ramified cover, unramified along $\mathring{E}_J$,
with $N_J:=\gcd(N_j,\ j\in J)$ sheets, and the ramifications indices are defined by the other $N_j$'s.
\end{proposition}

\begin{proof}
Let us pick-up a point $P\in\mathring{D}_J$. We can choose local coordinates $x_1,\dots,x_{n+1}$ centered at~$P$
such that if $J=\{j_1,\dots,j_s\}$, then the local equation of $D_{j_i}$ is $x_i=0$. Let us denote 
$P_Y\in Y$ the only point in $g^{-1}(P)$. The local equation of $Y$ on $P_Y$ is 
$t^N-\prod_{i=1}^{s}x_i^{N_i}=0$. The normalization process has two parts:
separate the local irreducible components and \emph{add} holomorphic forms. Only the first one affects
to the topology of $D_J$. Since $N$ is a multiple of all the $N_j$'s, $Y$ has $N_J$ irreducible components
and there is an obvious cyclic monodromy. The ramifications are produced along $D_{J'}$, with $J\subsetneqq J'$,
and the number of points depend on $N_j$ for $j\in J'$.
\end{proof}

\begin{cnt}\label{cnt:ss}
Steenbrink spectral sequence\index{spectral sequence!Steenbrink} is perfectly explained in the original paper~\cite{Steenbrink:77}. Let us recall what we need
about spectral sequences and weight filtrations. 
Let us consider a complex of vector spaces $\{(C_n,d_n)\}_{n\in\mathbb{Z}}$, $d_n:C_n\to C_{n+1}$, and let us consider 
a decreasing filtration $L_k(C_n)$ such that $d_n(L_k(C_n))\subset L_k(C_{n+1})$. 

This filtration induces a graduation and we can define
\[
\mathcal{C}^0_{p,q}:=\frac{L_q(C_{p+q})}{L_{q+1}(C_{p+q})}=\grArtal_q^L C_{p+q}
\text{ (the }0\text{-level of the spectral sequence)},
\]
which is also a complex with differentials $d_{p,q}^0:E^0_{p,q}\to E^0_{p+q,q}$. 

Let us define for $r\geq 1$:
\begin{align*}
Z^r_{p,q}:=&\frac{d_{p+q}^{-1}(L_{q+r}(C_{p+q+1}))+L_{q+1}(C_{p+q})}{L_{q+1}(C_{p+q})},\\ 
B^r_{p,q}:=&\frac{d_{p+q-1}(L_{q-r+1}(C_{p+q-1}))+L_{q+1}(C_{p+q})}{L_{q+1}(C_{p+q})},\\
\mathcal{E}^r_{p,q}:=&\frac{Z^r_{p,q}}{B^r_{p,q}},\text{ the }r\text{-level of the spectral sequence.}
\end{align*}
The actual cocycles are the elements in $\ker d$; the elements of $Z^r_{p,q}$ are represented by
the elements in $L_q(C_{p+q})$
such that their images by $d_{p+q}$ live in $L_{q+r}(C_{p+q+1})$, i.e.,
the image by the differential \emph{approaches} to $0$.
Something similar can be said for the coboundaries and the elements in $B^r_{p,q}$: they are represented
by the images by $d_{p+q-1}$ of elements in $L_{q-r+1}(C_{p+q-1})$ and their images \emph{approach} to the total set of boundaries.

There is a natural map $d^r_{p,q}:\mathcal{E}^r_{p,q}\to\mathcal{E}^r_{p-r+1,q+r}$ induced by the original differential.
By its very definition, $\mathcal{E}^{r+1}_{p,q}$ is exactly the cohomology defined by the differentials $d^r$,
i.e., the ${r+1}$-level is the cohomology of the $r$-level.
We can also define, 
\begin{align*}
Z^\infty_{p,q}:=&\frac{\ker d_{p+q}+L_{q+1}(C_{p+q})}{L_{q+1}(C_{p+q})},\\ 
B^\infty_{p,q}:=&\frac{d_{p+q-1}(C_{p+q-1})+L_{q+1}(C_{p+q})}{L_{q+1}(C_{p+q})},\\
\mathcal{E}^\infty_{p,q}:=&\frac{Z^\infty_{p,q}}{B^\infty_{p,q}},\text{ the }\infty\text{-level of the spectral sequence.}
\end{align*}
It is important to notice that $\mathcal{E}^\infty_{p,q}$ is the associated graded of $H^{p+q}(C^\bullet)$ by the filtration
induced by $L$. Hence, we can interpret a spectral sequence as a cohomological path starting from the 
cohomology of the induced graded complex and ending in the associated graded of the cohomology of the complex; we
say that the spectral sequence converges to the cohomology of the complex $C$.
Note that if after some $r$-level all the differentials vanish then $\mathcal{E}^r=\mathcal{E}^{r+1}=\dots=\mathcal{E}^\infty$
and we say that the spectral sequence \emph{degenerates} at $\mathcal{E}^r$.
\end{cnt}

\begin{theorem}[\cite{Steenbrink:77}]\label{thm:steenbrink}
There is a spectral sequence $\mathcal{E}$ converging to $H^*(\Sigma;\mathbb{Q})$, degenerating at $\mathcal{E}^2$.
\end{theorem}

Since its description is cumbersome we will present it only for the cases $n=1,2$, which are the cases we are
interested in. The generalization to the general case can be deduced with a little effort.
Let us fix the following notation:
\[
D^{(\ell)}=\coprod_{\#J=\ell + 1} D_J, \quad D_0^{(\ell)}=\coprod_{\#J=\ell + 1}^{0\notin J} D_J,
\]
where $\coprod$ stands for disjoint union; $D^{(\ell)}$ is also the normalization of the intersection
of $\ell +1$ divisors.
For $n=1$ we present the non-trivial $\mathcal{E}^1_{p,q}$ spaces and the non trivial differentials.
\begin{equation}\label{eq:sss1}
\begin{tikzcd}
H^0(D_0^{(1)};\mathbb{Q})&&\\
H^0(D_0^{(0)};\mathbb{Q})\uar&H^1(D_0^{(0)};\mathbb{Q})&H^2(D_0^{(0)};\mathbb{Q})\\
&&H^0(D_0^{(1)};\mathbb{Q})\uar
\end{tikzcd}
\end{equation}
We do the same thing for $n=2$ (we omit the coefficient ring):
\begin{equation}\label{eq:sss2}
\begin{tikzcd}[column sep=8mm]
H^0(D_0^{(2)})&&\\
H^0(D_0^{(1)})\uar["\delta_{0,1}"]&H^1(D_0^{(1)})&H^2(D_0^{(1)})\\
&&H^0(D^{(2)})\uar["\delta_{2,0}"]&&\\[-25pt]
&&\oplus&&\\[-25pt]
H^0(D_0^{(0)})\ar[uuu, "\delta_{0,0}"]&H^1(D_0^{(0)})\ar[uuu, "\delta_{1,0}"]&H^2(D_0^{(0)})&H^3(D_0^{(0)})&H^4(D_0^{(0)})\\
&&H^0(D^{(1)})\uar["\delta_{2,-1}"]&H^1(D^{(1)})\uar["\delta_{3,-1}"]&H^2(D^{(1)})\uar["\delta_{4,-1}"]\\
&&&&H^0(D^{(2)})\uar["\delta_{4,-2}"]
\end{tikzcd}
\end{equation}
The \emph{vertical} differentials at $\mathcal{E}^1$ are \emph{simplicial} combinations of restriction and Gysin maps between projective varieties
with quotient singularities ($V$-varieties\index{V variety@$V$-variety} in the language of Steenbrink) which are Hodge morphisms. The differentials at $\mathcal{E}^2$ vanish
and it is the reason of the convergence of the spectral sequence and the Mixed Hodge Structure can be somewhat visualized. 

In general, the only interesting elements of $\mathcal{E}^2=\mathcal{E}^\infty$ correspond to the case $p+q=n$.
We are going to define the \emph{vertical filtration\index{filtration vertical}} $W^\mathcal{E}$ of $H^n(\Sigma;\mathbb{Q})$ associated to the 
Steenbrink spectral
sequence as the one such that 
\[
\grArtal^{W^\mathcal{E}}_p H^n(\Sigma;\mathbb{Q})\cong \mathcal{E}^2_{p,n-p}.
\]
This is actually the filtration described earlier with a change of indices: $L_{n-p}=W^\mathcal{E}_p$ (in particular, it
is increasing).
In~\S\ref{sec:peso} we describe the weight filtration of a nilpotent morphism\index{nilpotent morphism}.
We show how Steenbrink relates these two filtrations.
From Theorem~\ref{thm:mon}, we have that $N_\Phi:=\Phi^N-\mathbf{1}$ is a nilpotent morphism.

\begin{theorem}[\cite{Steenbrink:77}]
Let $H^n_{\neq 1}(\Sigma;\mathbb{Q})\oplus H^n_{= 1}(\Sigma;\mathbb{Q})$
be the $\Phi$-invariant decomposition of $H^n(\Sigma;\mathbb{Q})$ such that:
\begin{enumerate}[label=\rm(\roman{enumi})]
\item $H^n_{\neq 1}(\Sigma;\mathbb{Q})$ does not have
$1$ as eigenvalue;
\item $\Phi$ is unipotent on $H^n_{= 1}(\Sigma;\mathbb{Q})$.
\end{enumerate}
Let us denote:
\begin{enumerate}[label=\rm(\alph{enumi})]
\item $W^{\neq 1}$:  the weight filtration of $N_\Phi$ restricted to $H^n_{\neq 1}(\Sigma;\mathbb{Q})$
and centered at $n$;
\item $W^{= 1}$:  the weight filtration of $N_\Phi$ restricted to $H^n_{= 1}(\Sigma;\mathbb{Q})$
and centered at $n+1$.
\end{enumerate}
Then, the \emph{vertical} filtration of $\mathcal{E}$ coincides with $W^{\neq 1}\oplus W^{= 1}$.
\end{theorem}

We obtain some consequences of this result.

\begin{corollary}
The Jordan form\index{Jordan form} of $\Phi$ is determined by the spectral sequence, more precisely, by the action
of $\Phi$ on $\mathcal{E}^2_{p,n-p}$.
\end{corollary}

We are going to be more concrete for $n=1,2$, though the same ideas can be generalized for arbitrary dimension.

\begin{corollary}\label{cor:Jordan1}
Let $(V,0)\subset(\mathbb{C}^2,0)$ be a germ of isolated singularity and let $\Phi$ be its monodromy.
\begin{enumerate}[label=\rm(\roman{enumi})]
\item The first column of \eqref{eq:sss1} is the cochain complex of the dual graph of $D$.
\item The first column on $\mathcal{E}^2$
is the cohomology of this graph; in particular, the graph is connected. 
\item There is a natural action of $\Phi$ on this dual graph, which determines the number of $2$-Jordan blocks for
each eigenvalue ($\neq 1$). More precisely,
\[
\Delta^{[1]}(t)=\frac{(t-1)\Delta_{\Phi,H^0(D^{(1)})}(t)}{\Delta_{\Phi,H^0(D^{(0)})}(t)}.
\]
\end{enumerate}
\end{corollary}

\begin{corollary}\label{cor:Jordan2}
Let $(V,0)\subset(\mathbb{C}^3,0)$ be a germ of isolated singularity and let $\Phi$ be its monodromy.
\begin{enumerate}[label=\rm(\roman{enumi})]
\item The first column of \eqref{eq:sss2} is the cochain complex of the dual complex of $D$.
\item The first column on $\mathcal{E}^2$
the cohomology of this complex; in particular, the $2$-complex must be connected and
with vanishing $H^1$. 
\item There is a natural action of $\Phi$ on this dual complex, which determines the number of $3$-Jordan blocks for
each eigenvalue ($\neq 1$). More precisely,
\[
\Delta^{[2]}(t)=\frac{\Delta_{\Phi,H^0(D^{(2)})}(t)\Delta_{\Phi,H^0(D^{(0)})}(t)}{(t-1)\Delta_{\Phi,H^0(D^{(1)})}(t)}.
\]

\item The maps $\delta_{2,0}, \delta_{3,-1}, \delta_{4,-1}$ are surjective; the maps $\delta_{1,0}, \delta_{2,-1}$ are injective.
These data determine the Jordan blocks of size~$2$:
\[
\Delta^{[1]}(t)=\frac{(t-1)^{\dim E^2_{4,-2}}\Delta_{\Phi,H^1(D^{(1)})}(t)}{(t-1)^{\dim E^2_{0,2}}\Delta_{\Phi,H^1(D^{(0)})}(t)}.
\]
\end{enumerate}
\end{corollary}

The condition on the normal blocks of maximal size can be expressed for any dimension. We consider the dual $n$-complex $\Gamma_D$
of the divisor $D$. We have that $\tilde{H}^j(\Gamma_D)=0$, if $j\neq n$ and the characteristic polynomial of $\Phi$ on 
$\tilde{H}^n(\Gamma_D)$ equals $\Delta^{[n]}$. The fact that this maximal bound is reached appeared in the same number of 
\emph{Invent. Math.} in two papers of N.~A'Campo~\cite{ACampo1973c} and B.~Malgrange~\cite{mg:73}.
 \section{Embedded \texorpdfstring{$\mathbb{Q}$}{Q}-resolutions}\label{sec:qres}

Thanks to Steenbrink's results, in order to obtain the Jordan form of the monodromy of an isolated hypersurface singularity
we need an embedded resolution of the hypersurface.
The main goal of \cite{ea:mams} was to provide such an embedded resolution of a SIS.
A.~Melle and me tried to use the same techniques to obtain an embedded resolution of a $k$-LYS  but we encountered technical difficulties. 

Note that the main ingredient of
Steenbrink spectral sequence is the (non-smooth) space of the semistable normalization. Actually the main goal in
\cite{Steenbrink:77} is to extend the results of \cite{Steenbrink:76} to the case of $V$-varieties\index{V variety@$V$-variety}.
Having these facts in mind, J. Mart{\'i}n-Morales~\cite{jmm:14} provided what is called an embedded $\mathbb{Q}$-resolution for a $k$-LYS,
where the Steenbrink spectral sequence can be constructed with the same properties. For SIS, his construction
follows the same basic ideas of~\cite{ea:mams} but the output is combinatorially much more simple.

Some background about cyclic abelian singularities can be found in \S\ref{sec:quot}.

\begin{definition}
Let $X$ be a $V$-variety\index{V variety@$V$-variety} of dimension~$n+1$, i.e. a complex variety with cyclic abelian singularities. Let $D$ be a divisor in $X$. We
say that $D$ is a\emph{ $\mathbb{Q}$-simple normal crossing divisor\index{Q simple normal crossing divisor@$\mathbb{Q}$-simple normal crossing divisor}} if $\forall P\in X$ there is a \emph{chart} $(\mathbb{C}^{n+1},0)\to(X,P)$
realizing the quotient singularity type of $X$ such that the preimage of $D$ is contained in the coordinate
hyperplanes of $\mathbb{C}^{n+1}$.
\end{definition}

\begin{definition}
Let $(V,0)\subset(\mathbb{C}^{n+1},0)$ be a germ of isolated hypersurface singularity.
An \emph{embedded $\mathbb{Q}$-resolution\index{Q resolution@$\mathbb{Q}$-resolution!embedded}} of $V$ is a proper birational map $\pi:(X,E)\to(\mathbb{C}^{n+1},0)$ which is an isomorphism outside the origin and
such that $X$ is a $V$-variety and $\pi^*V$ is a $\mathbb{Q}$-simple normal crossing divisor.
\end{definition}

The semistable normalization is defined in the same way and it is still a $V$-variety\index{V variety@$V$-variety} and $D$ is a $\mathbb{Q}$-simple normal crossing divisor~\cite{jmm:16}.
The main feature is that its combinatorics are much simpler that the ones of an embedded resolution.

It is a well-known result that the embedded resolution of plane curve singularities can be done by blowing-up points. We have stated in Theorem~\ref{thm:qres2} that an embedded $\mathbb{Q}$-resolution can be done using weighted 
blowing-up on points, see~\S\ref{sec:wbu} for details about weighted 
blowing-ups.

For hypersurface singularities in dimension~$2$, in general, one needs to blow up along curves. The hypersurface
singularities that can be solved by blowing-up points are called \emph{absolutely isolated singularities} (AIS).
They have been studied specially in problems related with complex singular foliations. 
They have been also useful to check some conjectures, e.g., A.~Melle proved  Durfee's conjecture~\cite{durfee:78} for AIS in \cite{MH00}. This conjecture states that $6$
times the geometric genus is not greater than the Milnor number.

It is very rare for $k$-LYS to be AIS; actually only when $\singArtal C_d=\emptyset$.
If it is not the case, the
embedded resolution for SIS of \cite{ea:mams} needs blowing-ups along points and smooth rational curves. 
Nevertheless, the method of \cite{jmm:14} provides both for SIS and $k$-LYS, a
$\mathbb{Q}$-resolution process where only weighted blowing-ups of points are needed. It seems reasonable to say that 
$k$-LYS are \emph{$\mathbb{Q}$-absolutely isolated singularities} and to study further features of these singularities.

Let us sketch the $\mathbb{Q}$-resolution process for $k$-LYS~\cite{jmm:14}. The first step is to blow-up the origin in $\mathbb{C}^3$.
This way, we obtain an exceptional divisor $E_1\cong\mathbb{P}^2$, and the singularities of $\hat{V}$ are in $\singArtal C_d$. We have seen this in Remark~\ref{rem:local_eq}. The surfaces $\hat{V}$
and $E_1$ are not transversal exactly at $\singArtal C_d$.

Let us fix $P\in\singArtal C_d$, which we assume to be $[0:0:1]$, in order to describe the part
of the $\mathbb{Q}$-resolution process over $P$. Let us recall that there is a chart where $E_1$ has equation
$z=0$ and $\hat{V}$ has equation $z^k+f_{d,P}(x,y)=0$. 

To simplify the notations
the strict transform of a divisor under a weighted blow-up will be denoted as the divisor,
the ambient space will make the difference from the divisor and its strict transform.

As stated in Theorem~\ref{thm:qres2}, there is a sequence of weighted blowing-ups 
which provide the minimal embedded $\mathbb{Q}$-resolution of $(C_d,P)$.
Let us assume that we start with a $(p,q)$-weighted blowing-up 
$\sigma^{P}_1:(Z^P_1,E^P_1)\to(E_1,P)$. Let $P_1,\dots,P_r$ the points in $E^P_1$
for which $(\sigma^{P}_1)^{-1}(C_d)$ is not a $\mathbb{Q}$-normal crossing divisor.  Let  
$m^P_1$
the $(p,q)$-weighted multiplicity of $C_d$ at $P$.

Let us consider now the situation in dimension~$3$ around $P\in E\subset\hat{\mathbb{C}}^3$.
We perform a $(kp,kq,m_1^P)$-weighted blowing-up 
$\tilde{\sigma}^P_1:(\tilde{Z},\tilde{E}^P_1)\to(\hat{\mathbb{C}}^3,P)$
with the above coordinates $(x,y,z)$. Actually:
\begin{itemize}
\item $E_1\cap\tilde{E}^P_1=E^P_1$,
\item 
$E_1\cap\hat{V}=C_d$, and
\item $E_1\cap\tilde{E}^P_1\cap\hat{V}=\{P_1,\dots,P_r\}$.
\end{itemize}
 
These are the points in where $E_1,\tilde{E}^P_1,\hat{V}$
do not form a $\mathbb{Q}$-normal crossing divisor (as in dimension~$2$).

We iterate the process, though we may find some technical difficulties, since $\tilde{Z}$
(and the subsequent spaces) may not be smooth, but weighted blowing-up are still valid.
This is one of the consequences.

\begin{proposition}
The topology of the part of the resolution process of $V$ after the first blowing-up is completely determined by the combinatorics
of $C_d$, including the multiplicities of the divisor.
\end{proposition}

The next step is to compare the divisors of the $\mathbb{Q}$-resolution and its semistable normalization for both $(V,0)$ and $(C_d,\singArtal C_d)$. There is a special irreducible component in the divisor of the semistable normalization, namely $D_1$ which is a $d$-cyclic cover of~$E_1$ ramified along the intersection
with the other irreducible components of the resolution. As stated in the previous proposition,
the rest of the components are completely determined by the combinatorics. We can summarize this discussion.

\begin{theorem}\label{thm:non_comb}
All the terms in \eqref{eq:sss2} (including its action by the monodromy) are determined by the combinatorics except the terms $H^j(D_1)$, factor of $H^j(D^{(0)})$, for $j=1,2,3$.
\end{theorem}

Let us look closer at the $V$-variety $D_1$ and its cohomology, together with the monodromy action.
The first important remark is that $V$-varieties have also Poincaré duality over~$\mathbb{Q}$. Recall
that $E_1$ has been obtained by several blow-ups of the original exceptional component
of the first blowing-up~$\sigma$ which was also denoted by $E_1$. Since it was naturally isomorphic
to the complex projective plane, and in order to avoid confusion, it is identified with~$\mathbb{P}^2$. The following commutative diagram holds:
\begin{equation}
\label{eq:cyclic_cover}
\begin{tikzcd}
D_1\rar["\tilde{\sigma}_D"]\dar["\tau_D"]&X_d\dar["\tau"]\\
E_1\rar["\tilde{\sigma}"]&\mathbb{P}^2
\end{tikzcd}
\end{equation}
where 
\begin{enumerate}[label=(CD\arabic{enumi})]
\item $\tau$ is the $d$-cyclic cover of $\mathbb{P}^2$ ramified along~$C_d$, see Construction~\ref{cnt:cyclic_cover};

\item $\tilde{\sigma}$ is the restriction to $E_1$ of the composition of all the weighted blowing-ups
of the $\mathbb{Q}$-resolution of $V$, except the blowing-up of the origin. It coincides with the minimal embedded
$\mathbb{Q}$-resolution of $C_d$ in $\mathbb{P}^2$;

\item the diagram coincides with the pull-back of $\tau$ and $\tilde{\sigma}$.

\end{enumerate}

We can derive some conclusions from this diagram. The monodromy $\eta$ of $X_d$ defined in Construction~\ref{cnt:cyclic_cover} lifts to the monodromy of $\tau_D$ induced by the monodromy
$\Phi$ of $V$. We are going to break the computation of the monodromy of $\eta_D$ on cohomology
in several steps. Since $\eta_D^d$ is the identity, it is enough to compute the characteristic polynomials
$\Delta^j(t):=\det\left(\mathbf{1}-t h^{*,H^j(D_1;\mathbb{Q})}\right)$ and combine them with the zeta function
\[
\zeta_{\eta_D}(t)=\prod_{j=0}^4 (-1)^{j}\Delta^j(t).
\]

This is the first consequence of the diagram.

\begin{paso}
The $V$-variety $D_1$ is birational to $X_d$, actually it is a model with only cyclic quotient
singularities.
\end{paso}
Since $X_d$ is connected we obtain the following step.

\begin{paso}
The $V$-variety $D_1$ is connected, i.e. $\dim H^0(D_1)=\dim H^4(D_1)=1$ and 
\[
\Delta^0(t)=\Delta^4(t)=1-t.
\]
\end{paso}

\begin{paso}
The $\zeta$ function $\zeta_{\eta_D}$ is combinatorial, so it is also the case for $\chi(D_1)$.
\end{paso}

\begin{proof}
Since $\tilde{\sigma}^{-1}(C_d)$ is a $\mathbb{Q}$-normal crossing divisor, $E_1$ admits a stratification
where over each stratum~$D$ the map $\tau_D$ is an unramified cyclic cover
(over each stratum the number of sheets is a divisor of $d$). 
Hence, there is a formula for $\chi(D_1)$ which is purely combinatorial. Actually,
a combinatorial formula also exists for~$\zeta_{\eta_D}$.
\end{proof}

\begin{paso}\label{paso:alexander}
The polynomial $\Delta^1$ is the Alexander polynomial $\Delta_{C_d}$ of $C_d$.
\end{paso}
\begin{proof}
Hodge theory is also valid on projective $V$-varieties; in particular 
$H^1(X_d;\mathbb{C})\cong H^1(X_d;\mathscr{O}_{X_d})\oplus H^0(X_d;\Omega^1_{X_d})$
and $\overline{H^1(X_d;\mathscr{O}_{X_d})}\cong H^0(X_d;\Omega^1_{X_d})$. Hence,
$H^1(X_d;\mathbb{C})$ is a birational
invariant and $\Delta^1$ is also the characteristic polynomial of $\eta$ on $H^1(X_d;\mathbb{Q})$,
which is, by Definition~\ref{dfn:alexander}, the Alexander polynomial of $C_d$.
\end{proof}

\begin{paso}
The combinatorics and the Alexander polynomial of $C_d$
completely determines the characteristic polynomials of $\eta_D$ on cohomology.
\end{paso}

\begin{proof}
Using Poincaré duality we have that $\Delta^1(t)=\Delta^3(t)$. Hence,
\[
\frac{\Delta^2(t)}{\Delta^1(t)^2}=\frac{\zeta_{\eta_D}(t)}{(1-t)^2}.
\]
The right hand-side is combinatorial, so the it is enough to compute $\Delta^1(t)=\Delta_{C_d}$.
\end{proof}

We end this section applying the previous discussion to the determination of the Jordan form
of the monodromy of $\Phi$ for a $k$-LYS. 
Let us start with the blocks of maximal size, i.e.~$3$
(only eigenvalues different from~$1$). The details can be found in~\cite[Théorème~4.3.1]{ea:mams}
and they follow from Corollaries~\ref{cor:Jordan1} and~\ref{cor:Jordan2}.

\begin{proposition}
For a SIS the structure of Jordan blocks is a combinatorial invariant, given $m\gg 1$
\[
\Delta^{[2]}(t)=\gcd\left((t-1)^m,\prod_{P\in\singArtal C_d}\Delta_P^{[1]}(t)\right).
\]
\end{proposition}

The same result holds for $k$-LYS
as it can be deduced from~\cite{jmm:14, jmm:16}. Note also that a necessary condition for the
existence of Jordan blocks of size~$3$ is that there are singularities in $C_d$ with Jordan blocks of
size~$2$. 

Let us consider now the Jordan blocks of size~$2$. 
\begin{theorem}
Let $V$ a $k$-LYS and let $\Delta_{C_d}(t)$ be the Alexander polynomial of the tangent cone.
\[
\Delta^{[1]}(t)=\frac{\Delta^{\cmbArtal,k}(t)}{\Delta_{C_d}(t)},
\]
where $\Delta^{\cmbArtal,k}(t)$ is a polynomial which depends only on the combinatorics and~$k$.
\end{theorem}

\begin{remark}
A precise formula appears 
in~\cite[Théorème~5.4.2 and Corollaire~5.5.4]{ea:mams}. Let us recall that 
up to a power of $t-1$ (which is combinatorially determined), we know that 
$\Delta^{[1]}(t)$ equals
\[
\frac{\Delta_{\Phi,H^1(D^{(1)})}(t)}{\Delta_{\Phi,H^1(D^{(0)})}(t)}.
\]
It is not complicated to compute the numerator once one has the data from the embedded $\mathbb{Q}$-resolution.
The connected components of $D^{(1)}$ are compact Riemann surfaces and the corresponding 
characteristic polynomials using $\zeta$ functions and Riemann-Hurwitz theory.

For the denominator, there are two kind of components. On one side we have $D_1$, and we have seen
in Step~\ref{paso:alexander} that its characteristic polynomial is $\Delta_{C_d}(t)$.
To have the specific formula  we need to compute the characteristic polynomials
for the other components.

We will not enter in the details but the key point for SIS is the following. Using 
\cite[Proposition~5.3.2]{ea:mams} we find that the each the maps $D_j\to E_j$, $j>1$
are birationally equivalent to maps $T\to\mathbb{P}^1\times\mathbb{P}^1$ which are ramified only along 
vertical and horizontal factors. Since on $H^1$ all these invariants do not depend on the birational
model we can compute the characteristic polynomials in a relatively easy way. 

There is a parallel work for $k$-LYS in~\cite{jmm:14} but in this case $\mathbb{P}^1\times\mathbb{P}^1$ is replaced
by a special ruled surface with quotient singularities; the ramification locus is composed by 
fibers and pairwise disjoint multi-sections. The computation of these characteristic polynomials
is more involved and it is the goal of~\cite{acm:24}. A concrete formula is expected in a forthcoming
paper of J. Mart{\'i}n-Morales.
\end{remark}

\begin{theorem}\label{thm:counter_yau}
There are singularities of hypersurface having the same link, the same characteristic polynomial
of the monodromy but which do not have the same embedded topology, i.e., Conjecture{\rm~\ref{cjt:yau}} does not hold.
\end{theorem}

\begin{proof}
It is enough to consider two SIS, or $k$-LYS with the same $k$, whose tangent cones
form an Alexander-Zariski pair, see~\cite{ea:jag} for the existence. Since the link and the characteristic polynomial
are combinatorial invariants they are the same. Since the Alexander polynomials
of the tangent cones differ, the structure of the Jordan form of the monodromy is different,
and then the embedded topology is different.
\end{proof}
 \section{Monodromy over \texorpdfstring{$\mathbb{Z}$}{Z}}\label{sec:z}

In the previous section, we have described the action of the monodromy of the Milnor fibration on the cohomology of the Milnor fiber
with coefficients in a field of characteristic~$0$, say $\mathbb{Q}$ or $\mathbb{C}$.
We have used an embedded resolution (or more precisely an embedded $\mathbb{Q}$-resolution) to obtain a model of the Milnor fiber and 
the geometric monodromy. With these ingredients and Steenbrink's and Milnor's theories, these computations have been achieved.

This method does not hold if we want to consider $\mathbb{Z}$ as ring of coefficients, much more finer ingredients are needed.
There is a theoretic classic method which involves morsifications, Picard-Lefschetz theory, and the intersection form of the Milnor fiber. There is an excellent exposition in~\cite[Part~I]{agv:88}.
Let us sketch this method. Let us start with a defining equation $F:(\mathbb{C}^{n+1},0)\to(\mathbb{C},0)$ of $V$ and let us consider a \emph{good} representative $F_{\varepsilon,\eta}:\overline{\mathbb{B}}_{\varepsilon,\eta}\to\overline{\mathbb{D}}^{2}_\eta$ as in~\eqref{eq:fibration},
which is a fibration outside the origin.

\begin{definition}
A \emph{morsification}\index{morsification} of $F$ is a holomorphic map $\mathfrak{F}$ defined on an open neighbourhood $\mathfrak{U}$ of  $\overline{\mathbb{B}}_{\varepsilon,\eta}\times\overline{\mathbb{D}}^{2}_r$, $r>0$, such that, if we define for $\lambda\in\overline{\mathbb{D}}^{2}_r$
\[
\overline{\mathbb{B}}_{\varepsilon,\eta,\lambda}:=\{\mathbf{p}\in\mathbb{C}^{n+1}\mid(\mathbf{p},\lambda)\in\mathfrak{U}, \abs{\mathfrak{F}(\mathbf{p},\lambda)}\leq\eta\},\quad
\begin{tikzcd}[row sep=0,/tikz/column 1/.append style={anchor=base east},/tikz/column 2/.append style={anchor=base west}]
\overline{\mathbb{B}}_{\varepsilon,\eta,\lambda}\rar["\mathfrak{F}"]&\overline{\mathbb{D}}^{2}_\eta\\
\mathbf{p}\rar[mapsto]&\mathfrak{F}(\mathbf{p},\lambda)
\end{tikzcd} 
,\text{ then}
\]
\begin{enumerate}[label=\rm(\alph{enumi})]
\item $F^0=F_{\varepsilon,\eta}$
\item If $\lambda\neq 0$ then $F^\lambda$ is a Morse function, i.e., all the singular points are non-degenerate and the singular values are pairwise distinct
(and in the open disk~$\mathbb{D}^2_\lambda$).
\end{enumerate}
\end{definition}

From classical Morse theory and standard compactness arguments we have the following statement.

\begin{proposition}
Morsifications of $F$ do exist. Moreover, after eventually restricting $r>0$:
\begin{enumerate}[label=\rm(\arabic{enumi})]
\item If $\lambda\neq 0$, then $F^\lambda$ has exactly $\mu$ singular points, where $\mu$ is the Milnor number of~$F$.

\item If $t\in\overline{\mathbb{D}}^{2}_\eta$ is a regular value of $F^\lambda$, then $(F^\lambda)^{-1}(t)$ is diffeomorphic to the Milnor fiber $\Sigma$ of $F$.

\item Each $F^\lambda$ defines a locally trivial fibration over $\overline{\mathbb{D}}^{2}_\eta\setminus\discArtal F^\lambda$ where $\discArtal$ stands for the set the critical values.

\item The restrictions of these fibrations to $\mathbb{S}^{1}_\eta$ are all equivalent.
\end{enumerate}
\end{proposition}

This proposition shows that we can compute the monodromy of $F$ by computing the monodromy of $F^\lambda$ over $\overline{\mathbb{D}}^{2}_\eta\setminus\discArtal F^\lambda$
and then restrict its result to  $\mathbb{S}^{1}_\eta$. This is where Picard-Lefschetz theory enters, see~\cite[Lemmas~1.3 and~1.4]{agv:88}.

\begin{proposition}\label{prop:pl}
Let $V$ be a non-degenerate singularity (i.e., after a change of coordinates, $F=x_0^2+\dots+x_n^2$). Then, $\Sigma$ is diffeomorphic to a closed regular neighbourhood
of the zero-section of the tangent vector bundle over $\mathbb{S}^n$. The image by $\sigma$ of the zero section in $\Sigma$ is called a \emph{vanishing cycle} and 
$(\sigma\cdot\sigma)^2_\Sigma=(-1)^{\frac{n(n-1)}{2}}(1+(-1)^n)$.
\end{proposition}

Fix a value $\lambda$, $0<\abs{\lambda}\leq r$. Let $\discArtal F^\lambda=\{t_1,\dots,t_\mu\}$ and let $t_0\in\mathbb{S}^1_\lambda$. We can choose a small radius $\delta>0$
such that the closed disks of radius $\delta$ centered at $t_i$ are pairwise disjoint and contained in the open disk $\mathbb{D}^2_\eta$ as in Figure~\ref{fig:morse}.
Let us choose points $s_i\in\mathbb{S}^1_{\delta}(t_i)$. In the fibers $(F^\lambda)^{-1}(s_i)$ we find a vanishing cycle $\hat{\sigma}_{i}$ corresponding to the critical
point~$t_i$. Fix a system of simple paths $\gamma_i$ in $\overline{\mathbb{D}}^2_\eta\setminus\bigcup_{j=1}^\mu\mathbb{D}^2_{\delta}(t_i)$, whose only intersection is their common origin $t_0$, which end at $s_i$ and such that from top to bottom, near $t_0$
they are ordered as $\gamma_1,\dots,\gamma_\mu$, see Figure~\ref{fig:morse}.

\begin{figure}[ht]
\centering 
\begin{tikzpicture}
\draw[->] (2,0) arc [start angle=0,end angle=180,
radius=2];
\draw (-2,0) node[left] {$\alpha_0$}  arc [start angle=180,end angle=360,
radius=2];
\fill[gray](2,0) node[right, black] {$t_0$} circle [radius=.1];
\foreach \x in {1,...,4}
{
\coordinate (A\x) at (30+60*\x:1.2);
\coordinate (B\x) at ($(A\x)+(.3,0)$);
\fill (A\x) circle [radius=.1];
\draw (A\x) node[left=5pt] {$t_\x$} circle [radius=.3];
\draw (2,0) -- node[above,pos={.5+\x/20}] {$\gamma_\x$} (B\x) -- (A\x);
}
\fill[gray] (B1) node[right=2pt, black] {$s_1$} circle [radius=.1];
\fill[gray] (B2) node[below right, black] {$s_2$} circle [radius=.1];
\fill[gray] (B3) node[above right, black] {$s_3$} circle [radius=.1];
\fill[gray] (B4) node[right, black] {$s_4$} circle [radius=.1];

\end{tikzpicture}
 \caption{Paths from the base value to the critical values, $\mu=4$.}
\label{fig:morse}
\end{figure}
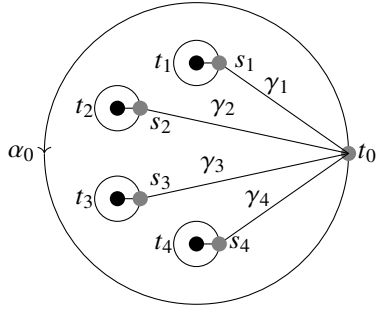

The fiber bundles over $\gamma_i$ are trivial, and then it is possible to transport $\hat{\sigma}_i\subset(F^\lambda)^{-1}(s_i)$ to 
$\sigma_i\subset(F^\lambda)^{-1}(t_0)=:\Sigma_\lambda$; recall that $\Sigma_\lambda$ is homeomorphic to $\Sigma$. Following 
\cite{agv:88} we have the following definition.

\begin{definition}
The tuple $(\sigma_1,\dots,\sigma_\mu)$ is a \emph{distinguished basis of vanishing cycles} of
the homology space $H_n(\Sigma_\lambda;\mathbb{Z})$.
\end{definition}

The name of basis is justified by the following result.

\begin{theorem}[{\cite[Theorem~2.1]{agv:88}}]
A distinguished basis of vanishing cycles is a basis of the abelian group $H_n(\Sigma_\lambda;\mathbb{Z})$.
\end{theorem}

In Proposition~\ref{prop:pl}, we pay attention to the intersection form in the Milnor fiber. Though this is a priori a weak invariant,
we are going to see that this is not really the case. Let us come back to the target of $F^\lambda$.
Note that
\[
\pi_1\left(\overline{\mathbb{D}}^2_\eta\setminus\bigcup_{j=1}^\mu\mathbb{D}^2_{\delta}(t_i); t_0\right)=
\langle\alpha_1,\dots,\alpha_\mu \mid \,\_\,\rangle
\]
where $\alpha_i:=\gamma_i\cdot\hat{\alpha}_i\cdot\gamma_i^{-1}$, and $\hat{\alpha}_i$ is the cycle based at $s_i$ that runs along 
$\mathbb{S}^1_{\delta}(t_i)$ counterclockwise. The previous choices guarantee that, up to homotopy,
$\alpha_0:=\alpha_1\cdot\ldots\cdot\alpha_\mu$ is the cycle based at $t_0$
that runs along 
$\mathbb{S}^1_{\eta}$ counterclockwise. 
The fact that $F^\lambda$ is a locally trivial fibration associates to each path $\alpha_i$
a monodromy diffeomorphism $\Phi_i:\Sigma_\lambda\to\Sigma_\lambda$ which induces
an automorphism $\Phi_{i,*}$ of $H_n(\Sigma_\lambda;\mathbb{Z})$.

\begin{proposition}[Picard-Lefschetz formula {\cite[Corollaries in page~26]{agv:88}}]
The monodromy $\Phi_0$ (on $\Sigma_\lambda$) is conjugate to the monodromy $\Phi$ on $\Sigma$.
Moreover, $\Phi_0=\Phi_{\mu,*}\circ\ldots\cdot\Phi_{1,*}$,
\[
\Phi_{i,*}(\sigma_j)=\sigma_j+(-1)^{\frac{(n+1)(n+2)}{2}}(\sigma_j\cdot\sigma_i)_{\Sigma_\Lambda}\sigma_i,
\]
and, hence, $\Phi_*$ can be computed in terms of the intersection form on a distinguished basis
of $H_n(\Sigma_\lambda,\mathbb{Z})$.
\end{proposition}

In \cite{agv:88}, one can learn much more, including the properties of the variation operator and the Seifert form.
The main problem about this method is that actual computations are really hard. The fact that there is a 
strong geometric knowledge
of the Milnor fibration in the case~$n=1$ makes very interesting the approach of 
A.M.~Gabri{\`e}lov in~\cite{gb:79}, since he proposed an inductive method 
on the dimension.

After a choice of good coordinates, the input  of Gabri{\`e}lov's method are the invariants of a \emph{generic} polar curve of 
$V$ and the intersection form of $(V_\beta,0)\subset(\mathbb{C}^n,0)$, where $V_\beta$ is the zero locus
of $F(\beta,x_1,\dots,x_n)$, for $0<\beta\ll 1$. We are not going to enter in the details, the polar curve
is a classical invariant of hypersurface singularities, see e.g.~\cite{Teissier:77}.
Starting from a distinguished basis of vanishing cycles for $V_\beta$ and its intersection form, using
the polar invariants Gabri{\`e}lov constructs a distinguished basis of vanishing cycles for $V$
and computes its intersection form.

In~\cite{Escario:06}, M.~Escario uses this method to provide an effective way to compute 
the intersection form of a $k$-LYS. Let us assume that the equation $f_d(x,y,z)=0$ of $C_d$ is given in some generic coordinates. In particular, the homogeneous polynomial $f_d(x,y,0)$ has no multiple roots.
As we pointed out, it is enough to
compute the intersection form for $F(x,y,z)=f_d(x,y,z)+z^{d+k}$.
Escario's procedure is quite technical but we want to emphasize that the input 
is the monodromy of the polynomial $f(x,y,1)$. This polynomial is a locally trivial fibration
outside the critical values, and this is the monodromy needed for  Gabri{\`e}lov's method

In~\cite{Escario:06a}
using a so-called \emph{discriminant method}, he computes the braid monodromy
of the discriminant of  $f(x,y,1)-t$ (with respect to the variable~$y$). This is principal part of the input in 
Gabri{\`e}lov's method. The relevant point is that this computation can be made effective using the
work of J.~Carmona~\cite{Carmona2003} and the joint work of M.~Marco and M.~Rodr{\'i}guez~\cite{mr:16}.

We are going to see how this method applies to $k$-LYS.
The first step is to present a problem stated (and solved) by O.~Saeki.

\begin{problem}[\cite{Saeki:91}]\label{pb:saeki}
Are there germs of isolated hypersurface singularities such that in its $\mu$-constant stratum
there is no germ admitting real equations?
\end{problem}
It is well-known that the answer is no if $n=1$. For $n=2$, the first step would be to study the problem
for the tangent cone (and looking for singularities where the combinatorics of the tangent cone is constant
in the $\mu$-constant stratum, which is not always the case, see~\cite{ABLM-milnor-number}). 

Saeki started from
the curve $(x^3-y^3)(y^3-z^3)(z^3-x^3)=0$. Of course, this is not a right example, since the previous equation
is actually \emph{real}. Nevertheless, it is not hard to see that in the space of curves with the same combinatorics,
it is not possible to get always real equations for the irreducible components $\ell_1=0,\dots,\ell_9=0$, all linear. 

The idea of O.~Saeki is to consider a non-reduced tangent cone $g_0:=\ell_1^{n_1}\cdot\ldots\cdot\ell_9^{n_9}$,
where the $n_i$ are positive and pairwise distinct. Of course, there is no SIS with this tangent cone, but 
Melle's $t$-singularities~\cite{MH00} can do the job. Let us consider 
a \emph{generic} homogeneous polynomial $g_1$ of degree~$n_1+\dots+n_9+1$,
namely, the projective curve $\{g_1=0\}$ intersects transversally the 
reduced curve $\{\ell_1\cdot\ldots\cdot\ell_9=0\}$.
Under these conditions,
the singularity defined by $F:=g_0+g_1$ is isolated and for any $\tilde{F}$
defining a singularity in its 
$\mu$-constant stratum,
$\tilde{F}$
does not admit real equations, since it can be proved that
the tangent cone is fixed.

Let us relate this problem with another one stated by N.~A'Campo. Let us recall that from the Monodromy Theorem~\ref{thm:mon}, the characteristic polynomial of $\Phi_{*}$ over $H^n(\Sigma;\mathbb{Q})$ is a product of cyclotomic polynomials. In particular,
$\Phi^{*}$ and $(\Phi^{*})^{-1}$ are conjugate since they have the same Jordan blocks.

\begin{problem}[\cite{acampo:03}]\label{pb:acampo}
Let $V$ be a germ of an isolated singularity of hypersurface. Let $\Phi$ be its monodromy and let $\Phi_*^\mathbb{Z}$
be the automorphism defined on $H_n(\Sigma;\mathbb{Z})$. Are $\Phi_*^\mathbb{Z}$ and $(\Phi_*^\mathbb{Z})^{-1}$ conjugate?
\end{problem}

The relationship between the two problems can from the following reasoning. Let $V$ be a germ of an isolated singularity of hypersurface with real equations. One can take the Milnor fiber over a real value. Note that the monodromy of $V$
is taken running counterclockwise a circle about the origin. The above condition 
implies that $V$ is preserved by the complex conjugation, which sends $\Phi$ to $\Phi^{-1}$. Hence, 
the answer to Problem~\ref{pb:acampo} is yes if $V$ is defined by a real equation, or more precisely, in its
$\mu$-constant stratum there is a singularity with real equations, e.g., if $n=1$.

\begin{remark}\label{remark:complex}
Actually in the previous discussion, one can give a sharper argument. If $V$ and $\overline{V}$ are in the same
$\mu$-constant stratum (even if no element of this stratum has real equations), then the answer to Problem~\ref{pb:acampo} is also yes. 
\end{remark}

Note that in Saeki's example~$V$, the singularities  $V$ and $\overline{V}$ are not in the same
$\mu$-constant stratum, so, it would be interesting to check its answer to Problem~\ref{pb:acampo}.
The main problem is the computation of the monodromy over the integers. There is an aside problem;
even if this monodromy can be computed, the Milnor number of $V$ is huge, so it will be difficult
to check the conjugacy. To be honest, in a private conversation, N.~A'Campo ensured that he had
effective methods for this checking, even for large ranks.

Thanks to M.~Escario's work an example of a $k$-LYS (or SIS) not satisfying 
the condition in Remark~\ref{remark:complex}
could lead to an answer to Problem~\ref{pb:acampo}.
In~\cite{alm:06} it was proved that two SIS were in the same $\mu$-constant stratum
if and only if their tangent cones had the same combinatorics and they were in the same connected
component of the space of curves with that combinatorics.

There is a combinatorics in~\cite[Example~4.2]{acm:07}, such that no curve with this combinatorial type has real equations 
and the space of curves with that combinatorics has two connected components, which are \emph{conjugate}.
The degree of the combinatorics is~$6$. Independently, M.~Namba~\cite{Namba1984} and A.I.~Degtyar{\"e}v~\cite{degt:89}
proved that the space of curves of degree $\leq 5$ with a given combinatorics is always connected.
Hence, the degree of this example is minimal.
 
\begin{proposition}\label{prop:complex}
The space of irreducible sextic curves having three singular points of types $\mathbb{E}_8$ ($v^3+u^5=0$),
$\mathbb{E}_7$ ($v^3+u v^3 + u^5=0$), and $\mathbb{A}_4$ ($v^2+u^5=0$) has two conjugate connected components.
Moreover, any such curve is projectively equivalent to either $C$ or $\overline{C}$
where $C$ is defined by the polynomial $f_6(x,y,z)$:
\begin{gather*}
\left(8 i + 4\right) x^{3} y^{3} + \left(-24 i - 12\right) x^{3} y^{2} z + \left(-24 i - 12\right) x^{2} y^{3} z + \left(24 i + 12\right) x^{3} y z^{2} + \\
\left(48 i + 4\right) x^{2} y^{2} z^{2} + \left(24 i + 12\right) x y^{3} z^{2} + \left(-8 i - 4\right) x^{3} z^{3} + \left(-24 i + 3\right) x^{2} y z^{3} +\\ \left(-24 i + 28\right) x y^{2} z^{3} + \left(-8 i - 4\right) y^{3} z^{3} + 5 x^{2} z^{4} + 20 i x y z^{4} - 20 y^{2} z^{4}.
\end{gather*}

\end{proposition}

\begin{proof}
In \cite{acm:07} we start with such a curve $C$. It is easy to prove that there is a unique projective automorphism such that the singular points are $[0:1:0], [1:0:0], [0:0:1]$ and 
that the tangent lines at the first two singular points are $x-z=0$ and $y-z=0$. We obtain then the equation must be either $f_6=0$ or $\overline{f}_6=0$. Moreover, the tangent line
at the last singular point has no real equation. If the curve would admit real equations, the coordinates of the three singular points are real (they are of different type, hence they cannot be complex conjugate), as the equations of the tangent lines. 
Then, the projective automorphism must be real and it is not possible.
\end{proof}

\begin{remark}\label{remark:ozp}
It is obvious that these curves $C,\overline{C}$ do not form a Zariski pair, since the complex conjugation in $\mathbb{C}$ defines a homeomorphism $(\mathbb{P}^2,C)\to(\mathbb{P}^2,\overline{C})$.
Note that this homeomorphism respects the natural orientation of $\mathbb{P}^2$ but does not respect the natural orientation of $C,\overline{C}$; see e.g.~\cite[p.~8]{agv:88} for the notion of complex orientation of a complex variety. These curves are natural candidates for the notion of \emph{oriented Zariski pairs}, i.e., curves with the same
combinatorics with no homeomorphism of pairs preserving orientations.

Unfortunately we do not know if $C,\overline{C}$ form an oriented Zariski pair. In~\cite{acm:07} we prove that if $C'$ is the curve of equation
$xz(125 x - 4 (2 + 11 i)z)=0$, then $C\cup C',\overline{C}\cup \overline{C'}$ form an oriented Zariski pair.
\end{remark}

The Milnor number of a SIS having $C$ as tangent cone is $144$. It is not known if it is minimal along the singularities of surface not satisfying the condition in 
Remark~\ref{remark:complex} but most likely the minimal Milnor number will be in the same order of magnitude.
There is a lot of work for these computations, but all the tools are available.

These ideas fall into the main features
of SIS: they provide sophisticated examples which have computable features.
 \section{SIS, \texorpdfstring{$k$}{k}-LYS, and Zariski pairs not of Alexander type}\label{sec:pb}

After the work of I.~Luengo, where he used SIS to prove that $\mu$-constant strata may not be smooth, using properties of projective plane curves,
the second main contribution of SIS was to provide counterexamples to Yau's Conjecture~\ref{cjt:yau}. The idea is the same one: use the properties
of the tangent cones. The counterexamples shown in Theorem~\ref{thm:counter_yau} come from the existence of Alexander-Zariski pairs, exhibited in~\cite{ea:jag}
and coming from the works of O.~Zariski.

Let us try to isolate the geometric properties that led to these results. A model of the Milnor fibration is obtained from divisor of the semistable
normalization of an embedded $\mathbb{Q}$-resolution. Let us consider to $k$-LYS $V',V''$ with tangent cones $C',C''$ having the same combinatorics.
We obtain two embedded $\mathbb{Q}$-resolutions with the same number~$r$ of irreducible components and a natural bijection
\[
\begin{tikzcd}[row sep=0,/tikz/column 1/.append style={anchor=base east},/tikz/column 2/.append style={anchor=base west}]
\{E'_J\mid\emptyset\neq J\subset\{0,1,\dots,r\}\}\rar[leftrightarrow]&
\{E''_J\mid\emptyset\neq J\subset\{0,1,\dots,r\}\}\\
E'_J\rar[leftrightarrow]& E''_J.
\end{tikzcd}
\]
This bijection induces homeomorphisms $h_J:E'_J\to E''_J$ if $J\neq\{1\}$. Moreover, these homeomorphisms are 
naturally compatible for distinct sets of indices.

As a consequence, if we consider the divisors $D',D''$ of the semistable normalizations, the bijection also induces compatible homeomorphisms $\tilde{h}_J$,
$J\neq\{1\}$,
such that the following diagram is commutative:
\[
\begin{tikzcd}
D'_J\rar["\tilde{h}_J"]\dar&D''_J\dar\\
E'_J\rar["h_J"]&E''_J.
\end{tikzcd}
\]

From the definition of Zariski pairs, we obtain the sufficient and necessary condition to obtain this diagram
also for $J=\{1\}$.

\begin{lemma}
There are homeomorphisms $h_1:E'_1\to E''_1$ and $\tilde{h}_1:D'_1\to D''_1$ extending the compatibility if and only the tangent cones do \emph{not}
form a Zariski pair.
\end{lemma}

We can interpret these results as the existence of a decomposition of the Milnor fibration. If the tangent cones are not in a Zariski pair, it is easy
to see that both Milnor fibrations are equivalent, and the singularities have the same embedded topology.

On the other side, if the tangent cones form a Zariski pair, then the decompositions are not equal and we cannot deduce that the Milnor fibrations are equivalent.
The decomposition of the Milnor fibration in terms of the semistable normalization have been thorough studied for $n=1$. Actually, the induced
the decomposition of the Milnor fiber coincides with Nielsen-Thurston~\cite{nie:44, th:88} decomposition of the monodromy, and this decomposition is unique up to isotopy!

If the unicity of this decomposition would hold in higher dimensions, we would deduce that if  $C',C''$ form a Zariski pair, then the $k$-LYS $V',V''$
would have different embedded topology. Actually, with some extra-work, if $V$ is a SIS with tangent cone the
curve $C\cup C'$ in Remark~\ref{remark:ozp}, it would provide a negative answer for Problem~\ref{pb:acampo}.

Even without this unicity, Theorem~\ref{thm:counter_yau} works if the tangent cones form an~\emph{Alexander-Zariski pair}, since  the Jordan form of the monodromies are distinct, and the same happens for the embedded topology.

\begin{problem}
Let  $V',V''$ be $k$-LYS whose tangent cones $C',C''$ form a Zariski pair which is not of Alexander type. What can be said about the embedded topology?
\end{problem}

The original examples of Zariski pairs were of Alexander type. The first Zariski pairs not of Alexander type appeared in~\cite{AC:98}. They correspond
to combinatorics of reducible curves. Actually, putting some weights on the connected components, they are distinguished by the Alexander polynomial
of some non-reduced curves. The problem is that up to now, we have not found how to translate this invariant into an invariant of a SIS or a $k$-LYS.

Much more types of Zariski pairs have been found, with different invariants that distinguish them, see a survey~\cite{act:08} of 2008, to check much more of them.
Actually, after this survey, qualitative different kinds of Zariski pairs have been found.

Let us point out several particular cases:
\begin{enumerate}[label=\rm(ZP\arabic{enumi})]
\item\label{zp1} Zariski pairs $C_1,C_2$ such that $\pi_1(\mathbb{P}^2\setminus C_1)\cong\pi_1(\mathbb{P}^2\setminus C_2)$.
\item Zariski pairs $C_1,C_2$ where $\mathbb{P}^2\setminus C_1\cong \mathbb{P}^2\setminus C_2$ but no homeomorphism
can be extended to $\mathbb{P}^2$.
\item\label{zp3} Zariski pairs $C_1,C_2$ defined by equations $F_i=0$, such that $F_i\in\mathbb{K}[x,y,z]$, where $\mathbb{K}$
is a number field, and there exists $\varphi\in\galArtal(\mathbb{K},\mathbb{Q})$ such that $F_2=F_1^\varphi$. They are
called \emph{arithmetic Zariski pairs} and they share all the algebraic properties. In the same way,
we can construct conjugated $k$-LYS
$V_1, V_2$ which share all the algebraic properties.
\end{enumerate}

All the above Zariski pairs are not of Alexander type. There are many examples in the literature, we sketch a family
which serves for~\ref{zp1} and~\ref{zp3}.

\begin{example}[{\rm\cite{acm:20}}]
Let us consider the combinatorics $\cmbArtal_n$, $n\geq 3$, defined by the following properties. There are four irreducible components, three lines and a curve of degree~$2n$.
This curve has three singular points all of them with the topological type of $v^n-u^{n+1}=0$, and the lines pairwise connect the singular points.
We prove in~\cite{acm:20} that after a projective change of coordinates the equation of such a curve can be written as $xyzg_{2n}(x,y,z)=0$ with the following properties:
\begin{enumerate}[label=(\alph{enumi})]
\item The monomials $x^ay^bz^c$ in $g_{2n}$ with non-vanishing coefficient satisfy the condition $\min(a+b,b+c,c+a)\geq n$.
\item The sum of the terms corresponding to the monomials $x^ny^bz^c$ equals $(y+z)^n$.
\item The sum of the terms corresponding to the monomials $x^ay^nz^c$ equals $(x+z)^n$.
\item\label{raiz} The sum of the terms corresponding to the monomials $x^ay^bz^n$ equals $(x+\zeta y)^n$, where $\zeta^n=1$.
\item The polynomial is \emph{generic} satisfying the above conditions.
\end{enumerate}
Two curves in $\cmbArtal_n$ are in the same connected component of $\cmbArtal_n$ if and only if the roots in \ref{raiz} are either equal or conjugate.

For $n=3$, we have two connected components, associated to $1,\exp\frac{2i\pi}{3}$. They form a Zariski pair since the fundamental groups of the complement are not isomorphic.
For $n>3$, the fundamental groups of the complements are always abelian. There is a sort of linking invariant that shows they form Zariski pairs (actually tuples).
For $n>3$, we are in \ref{zp1}, and for $n>4$, we have also examples of \ref{zp3}.
\end{example}

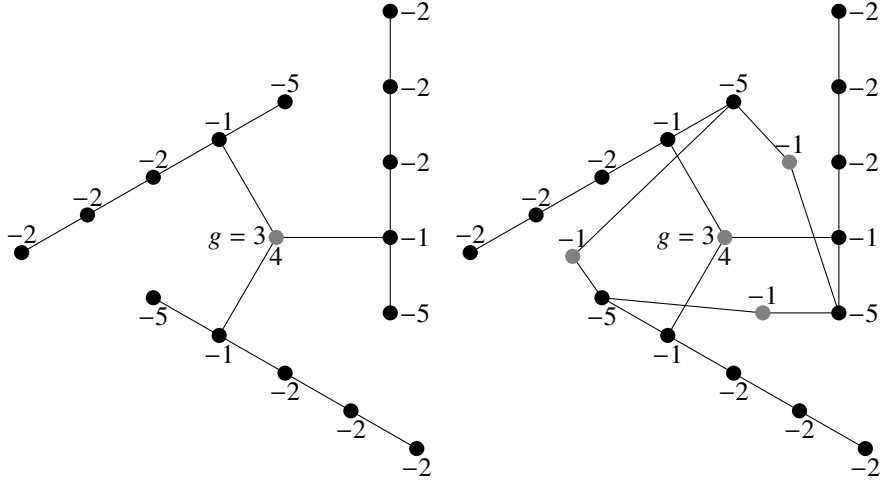
\begin{figure}[ht]
\centering
\begin{tikzpicture}
\begin{scope}
\coordinate (O) at (0,0);

\foreach \y in {0,1,2}
{
\foreach \x in {1,...,5}
{
\coordinate (A\x\y) at ([rotate={\y*120}] 1.5,{\x-2});
\fill (A\x\y) circle [radius=.1cm];
}
\draw (A1\y) -- (A5\y);
\draw (O) -- (A2\y);
}
\fill[gray] (O) node[below,black] {$4$} node[left,black] {$g=3$} circle
[radius=.1cm];

\foreach \x/\y in {1/5,2/1,3/2,4/2,5/2}
\foreach \z/\t in {0/right, 1/above, 2/below}
{
\node[\t] at (A\x\z) {$-\y$};
}
\end{scope}

\begin{scope}[xshift=5.9cm]
\coordinate (O) at (0,0);
\coordinate (B0) at (0.85,1);
\coordinate (B2) at (.5,-1);
\coordinate (B1) at (-2,-.25);
\foreach \y in {0,1,2}
{
\foreach \x in {1,...,5}
{
\coordinate (A\x\y) at ([rotate={\y*120}] 1.5,{\x-2});
\fill (A\x\y) circle [radius=.1cm];
}
\draw (A1\y) -- (A5\y);
\draw (O) -- (A2\y);
}

\foreach \x/\y in {1/5,2/1,3/2,4/2,5/2}
\foreach \z/\t in {0/right, 1/above, 2/below}
{
\node[\t] at (A\x\z) {$-\y$};
}

\draw (B0) -- (A10) -- (B2) -- (A12) -- (B1) -- (A11) -- cycle;

\fill[gray] (O) node[below,black] {$4$} node[left,black] {$g=3$} circle
[radius=.1cm];
\fill[gray] (B0) node[above,black] {$-1$} circle [radius=.1cm];
\fill[gray] (B1) node[above,black] {$-1$} circle [radius=.1cm];
\fill[gray] (B2) node[above,black] {$-1$} circle [radius=.1cm];
\end{scope}
\end{tikzpicture}
 \caption{$\cmbArtal_4$ (right) and $\cmbArtal_4^{\text{Irr}}$ (left)}
\label{fig:triangular}
\end{figure}

\begin{remark}
If we consider the combinatorics $\cmbArtal_n^{\text{Irr}}$ of the irreducible curves in $\cmbArtal_n$, we have also a disconnected space. For $n=3$, we have actually Alexander-Zariski pairs.
For $n>3$, it is unknown if they form a Zariski pair. It would be interesting if the corresponding SIS
may have different embedded topology and this property could be translated into a topological property of the curves in order to show that they form a Zariski pair.
\end{remark}
 \section{Weighted L{\^e}-Yomdin singularities}\label{sec:wlys}

Classical $k$-LYS allow to transform knowledge from complex projective plane curves to a family of germs of isolated singularities in $(\mathbb{C}^3,0)$
and the other way around. Following this idea in~\cite{ABLM-milnor-number} weighted L{\^e}-Yomdin singularities
play the same role where 
complex projective plane curves are replaced by projective curves in a weighted projective plane. 

The idea is similar. Fix a weight $\omega:=(p, q, r)$, $\gcd\omega=1$. Let $(V,0)$ be a germ of isolated singularity in $(\mathbb{C}^3,0)$
defined by $F(x,y,z)\in\mathbb{C}\{x,y,z\}$. As in \eqref{eq:desc_hom}, for some $k>0$, we can consider
\begin{equation}\label{eq:desc_whom}
F(x,y,z)= F_d(x,y,z) + \sum_{m\geq d+k} F_m(x,y,z),\qquad F_d\neq 0,
\end{equation}
but now instead of the decomposition
of~$F$ in homogeneous forms, we consider the one in $\omega$-weighted homogeneous forms. The subindex is the $\omega$-degree. We can consider in this case $d$ as the \emph{$\omega$-multiplicity}
of~$F$ or~$V$. 

We consider $C_m:=V_{\mathbb{P}^\omega}(F_m)$ which is a closed curve in the \emph{orbifold} $\mathbb{P}^2_\omega$, see~\S\ref{sec:wbu}. This is a
\emph{weighted projective plane curve}.

We want to consider a family of germs with the following property. Let $\sigma_\omega:(\hat{\mathbb{C}}^3_\omega, E)\to(\mathbb{C}^3,0)$
be the $\omega$-weighted blow-up of the origin,
see~\S\ref{sec:wbu}. Recall that $E\cap\hat{V}=C_d$, where
~$E\equiv\mathbb{P}^2_\omega$ is the exceptional divisor and $\hat{V}$ is the strict transform.
We want that outside $\singArtal(C_d)$ this intersection has $\mathbb{Q}$-normal crossings, and at $\singArtal C_d$ we have a situation similar to \ref{loc4}.

Following Definition~\ref{dfn:lys}, we may consider the condition $\singArtal(C_d)\cap C_{d+k}=\emptyset$. This is basically what we want but since
the total space of the $\omega$-weighted blowup is singular, we have to be more precise, as in~\cite{ACMNemethi}.

\begin{definition}\label{dfn:wlys}
Let $\omega:=(p, q, r)$, $\gcd\omega=1$, and $k>0$. Let $(V,0)$ be a germ of singularity in $(\mathbb{C}^3,0)$
defined by $F(x,y,z)\in\mathbb{C}\{x,y,z\}$ which can be decomposed as in~\eqref{eq:desc_whom}. Then we say 
that $(V,0)$ is an \emph{$(\omega,k)$-weighted L{\^e}-Yomdin singularity} ($(\omega,k)$-WLYS) if the following points
do not belong to $C_{d+k}$.

\begin{enumerate}[label=\rm(\roman{enumi})]
\item $P\in\singArtal(C_d)\setminus\singArtal(\hat{\mathbb{C}}^3_\omega)$.

\item\label{dfn:wlys-2} $P=[0:0:1]$, $r>1$, and the multiplicity of $F(x,y,1)$ is $>1$.
Something similar holds for the other vertices.

\item\label{dfn:wlys-3} $P\in C_d\cap\{z=0\}$, $P\neq [1:0:0], [0:1:0]$, $\gcd(p,q)>1$, and the intersection is not transversal at~$P$.
Something similar holds for the other axes.
\end{enumerate}
\end{definition}

As for $k$-LYS, it is not hard to prove that an $(\omega,k)$-WLYS is isolated.

\begin{remark}
There is a  geometric interpretation for \ref{dfn:wlys-2}. Let us consider 
the minimal resolution
of $(\mathbb{P}^2_\omega,P)$ of Proposition~\ref{prop:min_res}, with its exceptional divisor $E:=E_1\cup\dots\cup E_r$. Let $\hat{C}_d$ be the strict transform of $C_d$. The multiplicity of $F(x,y,1)$ is~$1$ if and only
if
\begin{enumerate}[label=\rm(\alph{enumi})]
\setcounter{enumi}{22}
\item $\hat{C}_d\cap E$ has only one point $\hat{Q}$ which is smooth for both $\hat{C}_d$ and $E$, with distinct tangent directions, and
$Q\in E_1\cup E_r$.
\item If $Q$ is in the strict transform of $x=0$, then $\gcd(q,r)=1$.
\item If $Q$ is in the strict transform of $y=0$, then $\gcd(p,r)=1$.
\end{enumerate}
For \ref{dfn:wlys-3}, if $\gcd(p,q)>1$, then for each $P=[x_0:y_0:0]_\omega$, for which either
\begin{enumerate}[label=\rm(\alph{enumi})]
\item $x_0y_0\neq 0$, or
\item $P=[0:1:0]_\omega$ and $\gcd(p,r)=1$, or
\item $P=[1:0:0]_\omega$ and $\gcd(q,r)=1$, then
\end{enumerate}
The singularity $(\hat{\mathbb{C}}^3_\omega,P)$ is cyclic quotient by a group of order $\gcd(p,q)$ and this is why the transversality condition
is needed.
\end{remark}

\begin{remark}
In the non-weighted case, given a curve $C_d$ we can construct $k$-LYS with $C_d$ as tangent cone for any $k>0$. In the weighted
case when there are singular points of $\mathbb{P}^2_\omega$ in $C_d$, it may happen that some $k$ cannot happen, since for some degrees a singular point 
of $E$ can be forced to be in any curve of that degree.
\end{remark}

\begin{example}
Let $\omega:=(2, 2, 3)$ and $C_{16}:=\{y^2 z^4 - x^5 z^2 + y^8=0\}$. 
Following \eqref{eq:orbifold}, there is an \emph{isomorphism}
\[
\begin{tikzcd}[row sep=0,/tikz/column 1/.append style={anchor=base east},/tikz/column 2/.append style={anchor=base west}]
\mathbb{P}^2_\omega\rar["\psi"]&\mathbb{P}^2_\eta\\
{[x:y:z]_\omega}\rar[mapsto]&{[x:y:z^2]_\eta}
\end{tikzcd}
\]
where $\eta=(1, 1, 3)$. Note that $\Phi(C_{16})=C_8=\{y^2 z^2 - x^5 z + y^8=0\}$. 
The only singular point of $\mathbb{P}^2_\eta$ is $[0:0:1]_\eta$ and $C_8\cap\{z=0\}=[1:0:0]_\eta$.
If we consider the affine chart $x=1$, the equation is $y^2 z^2 - z + y^8=0$, i.e, $C_8$
is smooth at $[1:0:0]_\eta$ but it is tangent to $z=0$.

Let us come back to the weight $\omega$. The space $(\hat{\mathbb{C}}^3_\omega$ is singular
at $P:=[0:0:1]_\omega\in C_{16}$ and at $\{z=0\}\subset E$. We have that $C_{16}\cap\{z=0\}=[1:0:0]_\omega=:Q$.
If we want to construct a $(\omega,k)$-WLYS, we need to take some $\omega$-weighted projective
curve $C_{16+k}=\{f_{16+k}(x,y,z)=0\}$ of $\omega$-degree $16+k$ such that $P,Q\notin C_{16+k}$.

This means that there are powers of  $z$ and $x$ with non-zero coefficients in $f_{16+k}(x,y,z)$, i.e.,
$16+k$ is a multiple of $6$. Hence the smallest $k$ equals~$2$. We can take
$f_{18}=x^{9}+z^6$. Then, $\{y^2 z^4 - x^5 z^2 + y^8 + x^{9}+z^6=0\}$ is an $(\omega,2)$-WLYS.
\end{example}

\begin{example}
Note that a singularity can be a $(\omega,k)$-WLYS for different values of~$(\omega,k)$.
This example was provided by I.~Luengo. Let $F_t(x,y,z)=x^5 + y^3 z + z^{11} + t x^2 y^2$,
$t\in\mathbb{C}$. 

We can consider it as a semiquasihomogeneous deformation of $F_0$. Note that $F_0$ defines
a smooth curve in $\mathbb{P}^2_{33,50,15}$ of degree~$165$. The weight of $x^2y^2$ is $166$, hence
the corresponding singularity is a $((33,50,15),1)$-WLYS. 

Consider the weight $(2,3,1)$. The weighted homogeneous polynomial $x^5 + y^3 z + t x^2 y^2$
has weighted degree~$10$ and then the corresponding singularity is a $((2,3,1),1)$-WLYS. 
It is also an example of $\mu$-constant family which is not $\mu^*$-constant family as the
well-known Brian{\c c}on-Speder example~\cite{bs:75}.
\end{example}

We can study $(\omega,k)$-WLYS using the same techniques as for $k$-LYS. In particular,
an embedded $\mathbb{Q}$-resolution is obtained as for $k$-LYS; it is enough to replace the first
standard blowing-up by an $\omega$-weighted blowing-up of the origin in~$\mathbb{C}^3$. Of course,
the total space has quotient singularities, but we already encountered this situation in the 
process to build an embedded $\mathbb{Q}$-resolution for $k$-LYS. The numerical invariants of these resolution
are trickier than for $k$-LYS, but the idea is the same one.

As one can define combinatorics of weighted projective curves, $(\omega,k)$-WLYS 
with $\omega$-\emph{weighted} tangent cones with the same combinatorics have embedded $\mathbb{Q}$-resolutions
which may differ only in the first divisor and the results of \S\ref{sec:qres} remain true with slights
modifications. Actually, the same happens for the results of other sections.
The notion of Zariski pairs (and Alexander-Zariski pairs) can be also defined, see~\cite{ACMNemethi}, providing also counterexamples
to Yau's conjecture~\ref{cjt:yau}.

Why  $(\omega,k)$-WLYS may be interesting? They provide a bridge between singularity theory and weighted
projective plane curves and this bridge is larger than the one provided by $k$-LYS. Let us state some cases.

\begin{enumerate}[label=(W\arabic{enumi})]
\item Examples of $\mu$-constant deformations which are not $\mu^*$-constant do exist among $(\omega,k)$-WLYS
and not among $k$-LYS. Moreover, it provides simple proofs for the topological triviality of these families
when the family of $\omega$-\emph{weighted} tangent cone is also topologically trivial.

\item A singularity of plane curve has a unique minimal embedded $\mathbb{Q}$-resolution, and most likely the
same happens for most $k$-LYS. The fact that a singularity can be $(\omega_i,k_i)$-WLYS, $i=1,2$,
for  $(\omega_1,k_1)=(\omega_2,k_2)$, may imply the existence of several minimal embedded $\mathbb{Q}$-resolutions.

\item There are $k$-LYS whose link is a $\mathbb{Q}$-homology sphere, see Proposition~\ref{prop:lys_qhs}.
In~\cite{ACMNemethi} we prove that there is no $k$-LYS whose link is a $\mathbb{Z}$-homology sphere, but there are among $(\omega,k)$-WLYS.

\item Monodromy conjecture for the topological and Igusa zeta functions holds for SIS~\cite{aclm:02}, and following the same ideas
it should hold for $k$-LYS. There are no yet formulas for the characteristic polynomial $\Delta$ and for the topological zeta function~$Z$
of an~$(\omega,k)$-WLYS but using~\cite{GLM:97} (for $\Delta$) and~\cite{lmvv:20} (for~$Z$) the computation
of such formulas is an affordable task, and
monodromy conjecture could be confronted.
\end{enumerate}

 \section{Quotient singularities}\label{sec:quot}

In this section we may a short presentation of abelian quotient singularities. There are several key references, e.g.~\cite{ca:57, Prill:67}. Roughly speaking, an (abelian) quotient singularity is the quotient 
of $(\mathbb{C}^n,0)$ by the action of a finite abelian subgroup~$G$ of analytic isomorphisms. To be more precise, 
such a $G$ acts also on $\mathscr{O}_{\mathbb{C}^n,0}=\mathbb{C}\{x_1,\dots,x_n\}$ and the fixed ring $\mathscr{O}_{\mathbb{C}^n,0}^G$ is normal
and defines a germ of complex variety, namely the corresponding quotient singularity.

Let us add a word of warning. This is closely related to \emph{orbifolds}, where one focuses not only on $\mathscr{O}_{\mathbb{C}^n,0}^G$, but also on the action. We are mostly interested in the varieties, but the quotient and orbifold language simplifies the exposition. We can simplify the exposition assuming that $G$ is a linear subgroup~\cite[Théorème~4]{ca:57}.

\begin{example}\label{ejm:dim1}
One could start with quotient singularities in dimension~$1$. It is not hard to see that a faithful action must be cyclic
and it will be by multiplication of roots of unity. If it is the action of $\mu_d:=\{\zeta\in\mathbb{C}\mid\zeta^d=1\}$,
then it is clear that $\mathbb{C}\{x\}^{\mu_d}=\mathbb{C}\{x^d\}$ and the corresponding \emph{quotient singularity}
is actually smooth.
\end{example}

D.~Prill~\cite{Prill:67} took care of this situation and defined the notion of \emph{small linear group}, the only reflection about a hyperplane is the identity. He showed that it is enough to consider small groups to get all the quotient singularities. We borrow some notation from~\cite{AMO:2014a} in order to describe these singularities.

\begin{ntc}
The quotient singularity of $(\mathbb{C}^n,0)$ defined by the action of 
the group $\mu_{d_1}\times\dots\times\mu_{d_r}$ given by
\[
(\zeta_1,\dots,\zeta_r)\cdot(x_1,\dots,x_n):=\left(x_1\prod_{j=1}^r\zeta_j^{a_{j1}},\dots,x_n\prod_{j=1}^r\zeta_j^{a_{jn}}\right)
\]
will be denoted as
\[
X\begin{pNiceArray}{c|ccc}
d_1&a_{11}&\cdots&a_{1n}\\
\vdots&\vdots&\ddots&\vdots\\
d_r&a_{r1}&\cdots&a_{rn}
\end{pNiceArray}.
\]
If $d_1=\dots=d_r=d$, we may denote it also by $\frac{1}{d}A$ where $A$ is the above $(a_{i,j})$ matrix
of size $r\times n$.
\end{ntc}
This notation also serves for the orbifolds.
Different symbols can represent the same singularity. Let us recall some operations which keep 
invariant the type of the quotient singularity.

\begin{lemma}[{\cite[Lemma 1.2]{AMO:2014b}}]
The following operations do not change the isomorphism type of
$X(\mathbf{d}| A)$.
\begin{enumerate}[label=\rm(\arabic{enumi}), series=ops]
\item Permutation of the columns of $A$ (the variables are correspondingly permuted).
\item Permutation of the rows of $(\mathbf{d}| A)$.
\item Multiplication of the last row of  $(\mathbf{d}| A)$ by a positive integer.
\item Multiplication of the last row of $A$ by an integer coprime with $d_r$.
\item Replace $a_{i,j}$ by  $a_{i,j}+kd_i$, $k\in\mathbb{Z}$.
\item If $d_r=1$, eliminate the last row.
\end{enumerate}
All these operations do not change the orbifold.
\end{lemma}

\begin{example}
With this lemma we can revisit Example~\ref{ejm:dim1}. It is not hard to check that 
we can assume $r=1$, and that $\frac{1}{d}(a)=\frac{\gcd(a,d)}{d}\left(1\right)$.
As orbifold, it is normalized, but as singularity, we end with the smooth germ.
\end{example}
There is a last operation which do not apply to orbifolds.

\begin{lemma}[{\cite[Lemma 1.2]{AMO:2014b}}]
The following operation does not change the isomorphism type of
$X(\mathbf{d}| A)$.
\begin{enumerate}[label=\rm(\arabic{enumi}), resume=ops]
\item If $e$ divides $\gcd(d_1,a_{1,1},\dots,a_{1,n-1})$ and $\gcd(e,a_{1,n})=1$
then replace $a_{i,n}$ by $ea_{i,n}$.
\end{enumerate}
In this case the isomorphism onto the new type is given by
\[
[(x_1,\dots,x_{n-1},x_n)]\mapsto[(x_1,\dots,x_{n-1},x_n^e)].
\]
\end{lemma}

With these operations, we can normalize $2$-dimensional quotient singularities, 
and we see that they coincide with Jung-Hirzebruch singularities, see~\cite{hnk:71}
for a nice description.

\begin{lemma}
Any $2$-dimensional quotient singularity is of the type $\frac{1}{d}(a,b)$, $\gcd(d,a)=\gcd(d,b)=1$.
Moreover, there $\alpha,\beta\in\mathbb{Z}$, $0<\alpha,\beta<d$, and coprime with~$d$
such that $\frac{1}{d}(a,b)=\frac{1}{d}(\alpha,1)=\frac{1}{d}(1,\beta)$.
\end{lemma}

Sometimes it is useful to present the singularities not in normal form to avoid cumbersome notation.
A minimal resolution of such a quotient singularity can be expressed using continuous fractions.

\begin{definition}
Let $q\in\mathbb{Q}_{>1}$. The \emph{continuous fraction decomposition} of $q$ is a tuple
defined recursively as follows. If $q\in\mathbb{Z}$, the tuple is $[q]$. If not, let $r:=\lceil q\rceil - q$.
If $[b_1,\dots,b_r]$ is the continuous fraction decomposition of $r_1^{-1}>1$, then 
$[\lceil q\rceil,b_1,\dots,b_r]$ is the continuous fraction decomposition of $q$.
\end{definition}8

For example, the continuous fraction decomposition of $\frac{n+1}{n}$ is $[2,\dots,2]$ with length~$n$.

\begin{proposition}[\cite{hnk:71}]\label{prop:min_res}
Let $(Z,P)$ be a quotient singularity of type $\frac{1}{d}(1,\beta)$, where $\gcd(d,\beta)=1$.
Let $[b_1,\dots,b_r]$ be the continuous fraction decomposition of $\frac{d}{\beta}$.

Let $\rho:(\hat{Z},E)\to(Z,P)$ be the minimal resolution of $(Z,P)$. 
Then $E=\rho^{-1}(P)$ is a linear divisor with $r$ irreducible components having as self-intersections
$-b_1,\dots,-b_r$ as in Figure{\rm~\ref{fig:jh}}. 

Let us assume that $Y\subset Z$ is a curve such that the local equation of $(Y,P)$ is $x_2=0$.
Let us denote also by $Y$ the strict transform of $Y$ by $\rho$. Then, $Y$ intersects transversally $E$
at a point in $E_1$ and $(Y\cdot Y)_{\hat{Z}}=(Y\cdot Y)_Z-\frac{\beta}{d}$.
\begin{figure}[ht]
\centering
\begin{tikzpicture}
\draw (-7/3,1/2) node[above] {$Y$} -- (-2/3,-1/2);
\draw (-5/3,-1/2) node[left] {$E_1$} -- node[below, pos=.7] {$-b_1$} (1/3,1/2);
\draw (-1/3,1/2) node[above] {$E_2$} -- ($.3*(4/3,-1/2) + .7*(-1/3,1/2)$);
\draw[dotted] ($.3*(4/3,-1/2) + .7*(-1/3,1/2)$) -- ($.6*(4/3,-1/2) + .4*(-1/3,1/2)$);
\draw ($.6*(4/3,-1/2) + .4*(-1/3,1/2)$) -- (4/3,-1/2) node[right] {$E_{r-1}$};
\draw (1/3,-1/2) -- node[above, pos=.6] {$-b_r$} (7/3,1/2) node[right] {$E_r$} ;
\end{tikzpicture}
 \caption{$\rho^{-1}(Y)$}
\label{fig:jh}
\end{figure}
\end{proposition}

\begin{example}\label{ejm:hnk}
In Remark~\ref{rem:abs_res_lys}, we stated that the normalization of the singularity defined by
$z^{k}-u^a v^b=0$ is a disjoint union of quotient singularities, as many as $\gcd(k,a,b)$.
Let us assume for simplicity that $\gcd(k,a,b)=1$. In that case it has the same normalization 
as $z^{k_1}-u^{a_1} v^b=0$, where $k_1=\frac{k}{\gcd(a,k)}$ and $a_1=\frac{a}{\gcd(a,k)}$.

Hence, again for the sake of simplicity we can assume that $\gcd(k,a)=\gcd(k,b)=1$.
If $0<c<k$ and $ac\equiv b\bmod{k}$, then $z^{k}-u^a u^b=0$ and $z^{k}-u v^c=0$
have the same normalization and the quotient singularity is $\frac{1}{k}(1,k-c)$.
\end{example}

 \section{Weighted blow-ups}\label{sec:wbu}

The weighted blow-ups are a generalizations of standard blow-ups, where the standard projective space is replaced 
by a weighted projective space, see~\cite{dol:82},  and are the bricks of embedded
$\mathbb{Q}$-resolutions. We restrict ourselves to the dimensions we are interested in.

Let $\omega=(p_0,p_1,\dots,p_n)\in\mathbb{Z}_{>0}^{n+1}$ such that $\gcd\omega=1$. The $\omega$-weighted projective space
$\mathbb{P}^n_\omega$ is the quotient of $\mathbb{C}^{n+1}\setminus\{0\}$ by the following action of $\mathbb{C}^*$:
\[
t\cdot(x_0,x_1,\dots,x_n):=(t^{p_0} x_0, t^{p_1} x_1,\dots, t^{p_n} x_n)
\]
with a natural structure of normal variety (sometimes an orbifold structure is preferred); 
the equivalence classes are denoted as $[x_0:x_1:\dots:x_n]_\omega$ which can be identified
with the closure of the orbit in $\mathbb{C}^{n+1}$.

\begin{example}
If $\gcd(p,q)=1$, then $\mathbb{P}^1_{(p,q)}$ is naturally isomorphic to $\mathbb{P}^1$, via $[x:y]\mapsto[x^p:y^q]$.
This isomorphism no longer holds if we treat $\mathbb{P}^1_{(p,q)}$ as orbifold.
\end{example}

\begin{remark}
As a generalization let $\omega=(p_0,p_1,\dots,p_n)\in\mathbb{Z}_{>0}^{n+1}$, $\gcd\omega=1$,
and such that $\gcd(p_1,\dots,p_n)=d>1$. Let $\eta:=\left(p_0,\frac{p_1}{d},\dots,\frac{p_n}{d}\right)$.
Then there is a natural isomorphism
\begin{equation}\label{eq:orbifold}
\begin{tikzcd}[row sep=0,/tikz/column 1/.append style={anchor=base east},/tikz/column 2/.append style={anchor=base west}]
\mathbb{P}^n_\omega\rar&\mathbb{P}^n_\eta\\
{[x_0:x_1:\dots:x_n]_\omega}\rar[mapsto]&{[x_0^d:x_1:\dots:x_n]_\eta}.
\end{tikzcd}
\end{equation}
\end{remark}

The weighted blow-ups can be defined using incidence varieties for the previous action as it is done for usual
blow-ups.

\begin{definition}\label{def:qbu}
Let $\omega=(p_1,\dots,p_n)\in\mathbb{Z}_{>0}^n$ such that $\gcd\omega=1$.
The \emph{$\omega$-weighted blow-up} of $(\mathbb{C}^n,0)$ is the map 
\[
\begin{tikzcd}[row sep=0,/tikz/column 1/.append style={anchor=base east},/tikz/column 2/.append style={anchor=base west}]
\{(\mathbf{p},\mathbf{q})\in\mathbb{C}^n\times\mathbb{P}^{n-1}_\omega\mid\mathbf{p}\in\mathbf{q}\}=:\hat{\mathbb{C}}^n_\omega\rar["\sigma_\omega"]&\mathbb{C}^{n}\\
(\mathbf{p},\mathbf{q})\rar[mapsto]&\mathbf{p}.
\end{tikzcd}
\]
\end{definition}

\begin{remark}
If $\gcd\omega:=d>1$, the $\frac{\omega}{d}$-weighted blowing-up is also called an $\omega$-weighted blow-up.
\end{remark}

As it happens with the usual blow-up, it is an isomorphism outside the origin, and the \emph{exceptional divisor}
$E:=\sigma_\omega^{-1}(0)$ is isomorphic to $\mathbb{P}^{n-1}_\omega$. The main difference is that 
the total space is no longer smooth, cyclic quotient singularities may arise. 

\begin{remark}
In a similar way, 
weighted blow-ups of quotient singularities can be defined, see~\cite{AMO:2014a, AMO:2014b}.
If $Z=X(\mathbf{d}|A)$ is a quotient singularity, $\mathbf{d}\in\mathbb{Z}_{>0}^r$, $<A\in\matArtal(n\times r;\mathbb{Z})$, 
and we perform an $\omega$-weighted blow-up of $(\mathbb{C}^n,0)$, the group acts naturally on it source, and its quotient
defines an $\omega$-weighted blow-up of~$Z$.
In this case
the singularities of the total space may be non-cyclic. 
\end{remark}

We describe 
singularities and intersections for some cases.

\begin{proposition}[{\cite[Theorem 4.3 and p. 14]{AMO:2014a}}]
Let $Z:=\frac{1}{d}(p,q)$, with $d,p,q$ pairwise coprime. Let $\sigma_{(p, q)}:\hat{Z}\to Z$ be the 
$(p,q)$-weighted blow up of $Z$ and $E$ its exceptional divisor. Then $E^2=-\frac{d}{pq}$
and $\hat{Z}$ has (at most) two singular points both in $E$ of type 
$\frac{1}{p}(-d,q)$ and $\frac{1}{p}(q,-d)$; moreover the local equation of $E$ at these singular points
corresponds to the coordinate for the weight $-d$.
\end{proposition}

We end the case of dimension~$2$, with the weighted version of resolution of singularities by blow-ups.

\begin{theorem}\label{thm:qres2}
Let $(C,0)\subset(\mathbb{C}^2,0)$ be a germ of isolated singularity. Let $\sigma:(X,E)\to(\mathbb{C}^2,0)$
be the minimal resolution of $(C,0)$. Let $r$ be the number of branching components in $E$
(i.e. those who intersect at least three other irreducible components of $\sigma^{-1}(C)$).
Then, there exists a composition $\sigma_w:(X_w,E_w)\to(\mathbb{C}^2,0)$ of a sequence 
of $r$ weighted blow-ups which is a $\mathbb{Q}$-embedded resolution of $C$.
\end{theorem}

In dimension~$3$, though the idea is the same the statements are cumbersome. To give some light, we describe 
the weighted blow-ups of smooth points.

\begin{proposition}
Let $\omega:=(p,q,r)$, where $\gcd\omega=1$.
Let $\sigma_\omega:(\hat{Z},E)\to(\mathbb{C}^3,0)$ be the 
$\omega$-weighted blow and $E$ its exceptional divisor. Then 
$\hat{Z}$ may have the following singularities in $E\equiv\mathbb{P}^2_\omega$
\begin{enumerate}[label=\rm(\alph{enumi})]
\item $[1:0:0]_\omega:\frac{1}{p}(-1,q,r)$;
\item $[0:1:0]_\omega:\frac{1}{q}(p,-1,r)$;
\item $[0:0:1]_\omega:\frac{1}{r}(p,q,-1)$;
\item at $x=0$ if $\gcd(q,r)>1$;
\item at $y=0$ if $\gcd(p,r)>1$;
\item at $z=0$ if $\gcd(q,r)>1$.
\end{enumerate}
\end{proposition}
 \section{Weight filtration of a nilpotent morphism}\label{sec:peso}

The conjugacy class of a nilpotent morphism $N:V\to V$, $V$ a vector space of finite dimension, is completely determined by its Jordan form. There is an equivalent and useful invariant called the \emph{weight filtration}, see e.g.~\cite{Schmid:73}.

\begin{definition}
Let $v\in V$. The \emph{$\alpha$-weight}, \emph{$\beta$-weight, and $\gamma$-weight} of $v$ are defined as 
$-\infty$ for $v=0$ and for $v\neq 0$:
\[
\alpha(v):=
\min\{\alpha\in\mathbb{Z}_{\geq 0}\mid N^\alpha(v)\neq 0\},\quad 
\beta(v):=-
\max\{\beta\in\mathbb{Z}_{\geq 0}\mid v\in N^\beta(V)\},
\]
and $\gamma(v):=\alpha(v)+\beta(v)$.
\end{definition}

The following lemma is straightforward.

\begin{lemma}
For $v_1,v_2\in V$ and $\omega=\alpha,\beta,\gamma$ we have 
\[
\omega(v_1+v_2)\leq\max(\omega(v_1),\omega(v_2)).
\]
Moreover, if $N(v_1)\neq 0$, then 
\[
\alpha(N(v_1))=\alpha(v_1)-1,\quad 
\beta(N(v_1))=\beta(v_1)-1,\quad
\gamma(N(v_1))=\gamma(v_1)-2.
\]
\end{lemma}

\begin{definition}
The \emph{weight filtration $W$ centered at $0$} of $N$ is defined as 
\[
W_k=W_k(V):=\{v\in V\mid \gamma(v)\leq k\}.
\]
\end{definition}
It is clear that $W_k$ is a subspace and it defines an increasing filtration. In order to understand better this filtration let us check it over 
a Jordan block.

\begin{example}
Let us assume that $\dim V=r$ and there is a basis $v_1,\dots,v_r$ such that $N(v_1)=0$
and $N(v_i)=v_{i-1}$, if $i>1$, i.e., $N$ has a Jordan block of size~$r$. From the definitions:
\[
\alpha(v_i)=i-1,\quad \beta(v_i)=i-r,\quad \gamma(v_i)=2i-(r+1).
\]
We have $(\gamma(v_1),\gamma(v_1),\dots,\gamma(v_r))=(-r-1,-r+1,\dots,r+1)$, i.e., it is 
a symmetric arithmetic sequence of ratio~$2$.
\end{example}

This example is the key for the following properties.

\begin{proposition}\label{prop:weight}
The weight filtration $W$ and the endomorphism $N$ satisfy:
\begin{enumerate}[label=\rm(\alph{enumi})]
\item $N(W_k)\subset W_{k-2}$.
\item For $k\geq 0$, the map $\grArtal_k^W\to\grArtal_{-k}^W$ induced by $N^k$ is an isomorphism.
\item \label{prop:weight_eigen} For $k\geq 0$, $\dim\grArtal_k^W=\dim\grArtal_{-k}^W$ is the number of Jordan blocks of size $k$
in Jordan form of $N$. 
\end{enumerate}

\end{proposition}

Sometimes it is convenient to shift this filtration.

\begin{definition}
Let $m\in\mathbb{Z}$; the \emph{weight filtration $W^{(m)}$ centered at $m$} of $N$ is defined as 
$W^{(m)}_k:=W_{k+m}$.
\end{definition}

%

 \providecommand\noopsort[1]{}
\providecommand{\bysame}{\leavevmode\hbox to3em{\hrulefill}\thinspace}
\providecommand{\MR}{\relax\ifhmode\unskip\space\fi MR }
\providecommand{\MRhref}[2]{%
  \href{http://www.ams.org/mathscinet-getitem?mr=#1}{#2}
}
\providecommand{\href}[2]{#2}

\end{document}